\documentclass[a4paper, 11pt, reqno]{amsart}
\usepackage[usenames,dvipsnames]{color}

\usepackage{enumitem}
\usepackage{verbatim}
\usepackage[noend]{algpseudocode}
\usepackage{algorithm}
\usepackage{hyperref}
\usepackage{amssymb, amsmath, latexsym, graphics, graphicx,caption}
\usepackage{tikz}
\usepackage{caption, subcaption}
\usetikzlibrary{calc}
\newtheorem{theorem}{Theorem}[section]
\newtheorem{lemma}[theorem]{Lemma}
\newtheorem{corollary}[theorem]{Corollary}
\newtheorem{proposition}[theorem]{Proposition}
\theoremstyle{definition}
\newtheorem{definition}[theorem]{Definition}
\newtheorem{example}[theorem]{Example}

\newtheorem{conjecture}[theorem]{Conjecture}
\newtheorem{remark}[theorem]{Remark}

\makeatother

\renewcommand{\iff}{\Leftrightarrow}

\def\N{{\mathbb N}}

\def\R{{\mathbb R}}

\newcommand\trop{\mathbb{T}}

\newcommand\ut[1]{\mathcal{UT}_{#1}(\trop)}
\newcommand{\Newt}{\mathsf{Newt}}
\newcommand\lab{\mathsf{label}}
\DeclareMathOperator{\conv}{conv}
\newcommand{\pluseq}{\mathrel{+}=}
\thanks{The authors would like to thank the Mittag-Leffler Institute and the organisers of the Spring 2018 program on Tropical Geometry, Amoebas and Polytopes for providing them with the opportunity to interact.}

\begin{document}
\title[Geometry of upper triangular tropical matrix identities]{Geometry and algorithms for upper triangular tropical matrix identities}
\author{MARIANNE JOHNSON AND NGOC MAI TRAN}
\date{\today}
\maketitle

\begin{abstract}
We provide geometric methods and algorithms to verify, construct and enumerate pairs of words (of specified length over a fixed $m$-letter alphabet) that form identities in the semigroup $\ut{n}$ of $n\times n$ upper triangular tropical matrices. In the case $n=2$ these identities are precisely those satisfied by the bicyclic monoid, whilst in the case $n=3$ they form a subset of the identities which hold in the plactic monoid of rank $3$. To each word we associate a signature sequence of lattice polytopes, and show that two words form an identity for $\ut{n}$ if and only if their signatures are equal. Our algorithms are thus based on polyhedral computations and achieve optimal complexity in some cases. For $n=m=2$ we prove a Structural Theorem, which allows us to quickly enumerate the pairs of words of fixed length which form identities for $\ut{2}$. This allows us to recover a short proof of Adjan's theorem on minimal length identities for the bicyclic monoid, and to construct minimal length identities for $\ut{3}$, providing counterexamples to a conjecture of Izhakian in this case. We conclude with six conjectures at the intersection of semigroup theory, probability and combinatorics, obtained through analysing the outputs of our algorithms.

\thanks{}
\end{abstract}

\section{Introduction}
Consider arithmetic over the tropical semifield $\mathbb{T} = (\mathbb{R}\cup\{-\infty\}, \oplus, \odot)$, where $x\oplus y = {\rm max}(x,y)$ and $x \odot y = x+y$. 
For each positive integer $n$, the set of $n \times n$ upper triangular matrices 
$$\ut{n} = \{M \in \mathbb{T}^{n \times n}: M_{ij} = -\infty \mbox{ if } i < j\}$$ 
is a semigroup under matrix multiplication. We write $w \sim_n v$ to denote that a pair of words $w,v$ over an alphabet $\Sigma$ is a semigroup identity for $\ut{n}$, meaning that each morphism $\varphi$ from the free semigroup on $\Sigma$ to $\ut{n}$ satisfies $\varphi(w)=\varphi(v)$. The study of semigroup identities for $\ut{n}$ has attracted much attention in recent literature \cite{CHLS16,I14, I14Erratum, IM10,O15, Shi14, TaylorThesis}, with interesting connections between the equational theory of these monoids and the bicyclic \cite{DJK17} and plactic \cite{Cain17,I17} monoids. Typical results in this direction are as follows. The equivalence relation determined by $\sim_n$ is a nonidentical relation \cite{I14} which is a refinement of the relation determined by identities for the bicyclic monoid \cite{IM10}; in the case $n = 2$ this refinement is trivial \cite{DJK17}; in the case $n=3$, the relation $\sim_3$ is also a refinement of the corresponding relation for the plactic monoid of rank three \cite{Cain17,I17}. Thus one can easily deduce, for example, that words in the same $\sim_n$ equivalence class must contain the same number of each letter. However, little insight into the structure of $\sim_n$ equivalence classes has been gained through analyzing known identities of $\ut{n}$. In part, this is because existing methods for constructing these identities require delicate arguments, developed with the primary focus of proving the existence of nontrivial identities \cite{I14, IM10, O15, TaylorThesis}.

In this work, we give new geometric methods and algorithms to verify, construct and enumerate identities in $\ut{n}$ amongst words with a given number of occurrences of each letter from a fixed alphabet. In general, verifying whether two given words form an identity in a given semigroup provides a major computational challenge. In the case of $\ut{n}$, the na\"{i}ve algorithm would be to multiply out the matrices symbolically, and then test for functional equality of the resulting pairs of tropical polynomials corresponding to each entry. Both of these operations are costly when the words are long. Our approach is based upon a geometric interpretation of the characterisation of $\ut{n}$ identities given in \cite{DJK17}, which involves more polynomial pairs but in fewer variables and monomials. A key observation is that the polynomials which appear in \cite{DJK17} are tropical polynomials with trivial coefficients, and thus they are equal if and only if their Newton polytopes are equal. This translates the problem to verifying equality of lattice polytopes, which can be efficiently computed and, more importantly, adds geometric insights that allow us to deduce structural information about these identities. 

Our strongest structural result applies to $\ut{2}$ identities over a two letter alphabet. Let $W(\ell_a,\ell_b)$ be the set of words with $\ell_a$ occurrences of $a$ and $\ell_b$ occurrences of $b$. This set is naturally in bijection with Northeast staircase paths from $(0,0)$ to $(\ell_a,\ell_b)$. In particular, $W(\ell_a, \ell_b)$ is a distributive lattice with respect to the natural partial order $\preceq$, in which $w \preceq v$ if and only if the path of $w$ does not rise above the path of $v$. The following Structural Theorem shows that $\ut{2}$ identities (or equivalently, identities for the bicyclic monoid \cite{DJK17}) respect the distributive lattice structure of $W(\ell_a,\ell_b)$. In particular, it implies that the $\sim_2$ equivalence class of $w \in W(\ell_a, \ell_b)$ is completely specified by the unique minimal and maximal words in this class. We give an explicit and efficient construction for these words in Theorem~\ref{thm:min.max}.
\begin{theorem}[Structural Theorem]\label{thm:chain}
Let $w,u, v \in W(\ell_a,\ell_b)$. 
\begin{itemize}
\item[(i)] If $w \preceq u \preceq v$ and $w\sim_2 v$, then $w\sim_2 u$.
\item[(ii)] If $w\sim_2 v$, then $w \sim_2 w\vee v \sim_2 w\wedge v$.
\end{itemize}
\end{theorem}

In general, for any finite alphabet, we provide 
algorithms to solve the problems listed below. 
\begin{itemize}
  \item \texttt{CheckPair}: given a pair of words $w,v$ over alphabet $\Sigma$ and a positive integer $n$, decide if $w\sim_n v$.
  \item \texttt{ListWord}: given a word $w$ over alphabet $\Sigma$ and a positive integer $n$, compute the equivalence class of $w$ in $\ut{n}$.
  \item \texttt{ListAll}: given an alphabet $\Sigma$, a vector $(\ell_1, \dots, \ell_{|\Sigma|})$ of non-negative integers, and a positive integer $n$, list all $\ut{n}$ equivalence classes of words with $\ell_i$ occurrences of the $i$-th letter in $\Sigma$. 
\end{itemize}
The worst-case complexity of our various algorithms is summarised in the following theorem. 
\begin{theorem}\label{thm:algorithm}
Let $m, n \geq 2$. For non-negative integers $\ell_1,\ldots, \ell_m$, let $W(\ell_1, \ldots, \ell_m)$ denote the set of words of length $\ell:=\ell_1 + \cdots + \ell_m$ containing exactly $\ell_i$ occurrences of letter $a_i$. For $w \in W(\ell_1, \ldots, \ell_m)$, let $C_n(w)$ denote the size of the $\sim_n$-equivalence class of $w$ and let $C_n(\ell)$ denote the number of $\sim_n$-equivalence classes of words in $W(\ell_1,\dots,\ell_m)$. The following table shows the worst-case complexity of our algorithms in various cases. 
\begin{center}
\begin{tabular}{ l|c|c|c}
Task & $m = 2, n = 2$ &  $m = 2, n > 2$ & $m > 2$ \\	
\hline
\texttt{CheckPair} & $\ell$ & $2^{n-1}\ell^2$  & $m^{n-1}{\ell \choose n-1}^2$  \\
\hline
\texttt{ListWord} & $\ell$ & 
$\sum_{d=2}^{n-1}2^d\binom{\ell}{d}^2C_d(w)$
& $m^{n-1}{\ell \choose \ell_{1}, \dots, \ell_{m}}{\ell \choose n-1}^2$ \\
\hline
\texttt{ListAll} & $C_2(\ell)
\ell^2$ & $2^{n-1}{\ell \choose \ell_1}{\ell \choose n-1}\ell^2$  & $m^{n-1}{\ell \choose \ell_{1}, \dots, \ell_{m}}{\ell \choose n-1}^2$ \\
\hline
\end{tabular}
\end{center}
Furthermore, the complexity for \texttt{CheckPair} and \texttt{ListWord} are the best possible for the case $m = n=2$.
\end{theorem}

To the best of our knowledge, the only other algorithmic approaches to studying $\ut{n}$ arise via the connection to the bicyclic monoid which satisfies the same identities as $\ut{2}$ by \cite{DJK17}). Shle\u\i fer \cite{Shl90} enumerated all identities of the bicyclic monoid of length at most $13$, by means of a computer program. Pastijn \cite{P06} gave an algorithm for identity verification in the bicyclic monoid  via linear programming. In the case of a two letter alphabet, Pastijn's algorithm has complexity $O(\ell^2)$, while ours attains the optimal complexity of $O(\ell)$ (cf. Section \ref{sec:alg}). 

Our algorithms are efficient enough to find the shortest identities that hold in $\ut{3}$ through exhaustive search (noting that it suffices to consider words over a two letter alphabet). Without our results, constructing \emph{shortest} identities is a prohibitive computation. For instance, to prove that there is no $\ut{3}$ identity of length 21, after grouping words by pairs that have the same number of each letter, there are still $\sum_{i=1}^{10} {{21 \choose i} \choose 2} > 1.34 \times 10^{11}$ pairs of words to consider. In comparison, our algorithm takes less than 30 minutes on a conventional laptop to complete this task.
\begin{theorem}\label{thm:shortest.ut3}
The shortest $\ut{3}$ identities have length $22$. Up to exchanging the roles of $a$ and $b$ and reversing the words, these are: 
\begin{eqnarray*}
abbaabba\;ab\;baababbbbaba &\sim_3&abbaabba \; ba\; baababbbbaba \\
abbabaabab\; ab\;babaababba &\sim_3&abbabaabab\;ba\;babaababba\\
ababbabaab\;ab\;baababbaba &\sim_3&ababbabaab\;ba\;baababbaba\\
abbabaabab\;ab\;ababbabaab &\sim_3&abbabaabab\;ba\;ababbabaab\\
ababbabaab\;ab\;ababbabaab &\sim_3&ababbabaab\;ba\;ababbabaab\\
abbaabbaba\;ab\;baababbaba &\sim_3&abbaabbaba\;ba\;baababbaba\\
abbaabbaba\;ab\;babaabbaab &\sim_3&abbaabbaba\;ba\;babaabbaab\\
ababbabaab\;ab\;abbabaabab &\sim_3&ababbabaab\;ba\;abbabaabab\\
ababbabaab\;ab\;babaababba &\sim_3&ababbabaab\;ba\;babaababba\\
abbaabbaba\;ab\;babaababba &\sim_3&abbaabbaba\;ba\;babaababba
\end{eqnarray*}
\end{theorem}

Theorem \ref{thm:shortest.ut3} disproves a conjecture of Izhakian \cite[Conjecture 6.1]{I14}, which in the case $n=3$
 predicts that a certain identity of length 26 is shortest. Moreover, since $\sim_3$ is a refinement of the relation defined by identities of $\mathcal{P}_3$, the plactic monoid of rank three \cite{Cain17, I17}, Theorem \ref{thm:shortest.ut3} represents the shortest identities currently known to hold in $\mathcal{P}_3$ (these will be shortest if the refinement turns out to be trivial). We note that Taylor \cite{TaylorThesis} has disproved  \cite[Conjecture 6.1]{I14} in the case $n=4$ by means of a carefully constructed example. The analogue of Theorem \ref{thm:shortest.ut3} for $n = 2$ corresponds to Adjan's result on minimal identities for the bicyclic monoid \cite{A66}. For ease of comparison, we restate Adjan's result and provide a new geometric proof in Section \ref{sec:main}.

Finally, our algorithms provide a wealth of data and insights, which lead to interesting  results and conjectures regarding the statistics of $\ut{n}$ equivalence classes. 
We present six such conjectures in our paper. Conjectures \ref{conj:adjacent} 
and \ref{conj:isoterm.local}
state that weaker versions of the Structural Theorem hold for $n \geq 3$. Conjectures \ref{conj:isoterm.fraction.1}, \ref{conj:fraction.ut3} and \ref{conj:isoterm.fraction.2} concern probabilistic constructions for long $\ut{n}$ identities on two-letter alphabets. Conjecture \ref{conj:large} gives a candidate for the largest $\sim_2$ class amongst words of fixed length $\ell$. We present further details, data and reasoning to support these conjectures in Section \ref{sec:alg-results}. 

Our paper is organised as follows. We review essential definitions and results in Section \ref{sec:background}, refining a result of \cite{DJK17} to show that $\sim_n$ can be characterised in terms of certain tropical polynomials. In Section \ref{sec:support} we show how to compute the support sets of these polynomials, and use this to provide appealing geometric interpretations of several properties of the equational theory of the bicyclic monoid. In Section \ref{sec:main}, we further specialise to the two-letter case, where we collect our major results. This includes a geometric proof of Adjan's theorem, proof of the Structural Theorem, and an explicit construction of the minimal and maximal elements of each $\ut{2}$ equivalence class (Theorem \ref{thm:min.max}). In Section \ref{sec:alg}, we give the algorithms and prove their complexity. We discuss the conjectures in conjunction with the results obtained through the algorithms in Section \ref{sec:alg-results}. Section \ref{conclusion} concludes the paper.
Documented code to reproduce all our experimental results is available at \url{https://github.com/princengoc/tropicalsemigroup}.

\subsection*{Notations}
We write $\mathbb{N}_{\geq 0}$ to denote the set of non-negative integers, and $\mathbb{N}$ to denote the set of positive integers. If $\Sigma$ is a finite alphabet, then $\Sigma^*$ (resp. $\Sigma^+$) will denote the \emph{free monoid} (resp. \emph{free semigroup}) on $\Sigma$, that is, the set of finite words (resp. finite nonempty words) over $\Sigma$ under the operation of concatenation. For a word $w \in \Sigma^+$, we write $w_i$ for the $i$-th letter of this word, $|w|_a$ for the number of occurrences of the letter $a$ in $w$ for each $a \in \Sigma$, and $|w| = \sum_{a\in\Sigma}|w|_a$ for the length of the word. Fixing a total order on $\Sigma$, the sequence $(|w|_a: a \in \Sigma)$ is called the \emph{content} of the word, denoted $c(w)$.  For each $d \in \mathbb{N}_{\geq 0}$, let $\Sigma^d = \{w \in \Sigma^*: |w| = d\}$. For each $n \in \mathbb{N}$ and each pair of words $w,v \in \Sigma^+$, write $w \sim_n v$ if $w=v$ is a \emph{semigroup identity} for $\ut{n}$. If $w$ is the unique element in its $\sim_n$-class, we say that $w$ is an \emph{isoterm} for $\ut{n}$. We write $w \leftrightarrow u$ to denote that $u$ is obtained from $w$ by a single adjacent letter swap, that is $w=w'\;xy\;w''$ and $u=w'\;yx\;w''$ where $w', w'' \in \Sigma^*$ and $x,y\in \Sigma$ with $x \neq y$. For $X, Y \subseteq \mathbb{R}^N$, we write $X \simeq Y$ if there exists an invertible affine linear transformation $f$ such that $f(X) = Y$. For a finite set of points $X$, let $\conv(X)$ denote the convex hull. Write $\mathcal{V}(P)$ for the set of vertices of a polytope~$P$.  For presentation of algorithms, $\leftarrow$ means variable assignment, $\pluseq$ means list addition. For a dictionary $L$, wew write $keys(L)$ to denote the list of its keys. If $u$ is not a key in a dictionary $L$, $L[u]$ is defined to be the empty list.

\section{Background}\label{sec:background}
In this section we recall the connection between tropical polynomials and identities of the semigroup $\ut{n}$. To each word $w \in \Sigma^+$ we then associate a signature, which is a sequence of lattice polytopes. The main result of this section is Theorem \ref{thm:reduce_checks}, which states that
$w \sim_n v$ if and only if the first $\sum_{j=1}^{n-1} |\Sigma|^j$ polytopes in their signatures are equal. This reduces identity verification in $\ut{n}$ to checking equality of a sequence of lattice polytope pairs, a key idea of our paper.

Let $I$ be a finite subset of $\mathbb{N}_{\geq0}^m$. The tropical polynomial with support $I$ and coefficients $(c_a \in \mathbb{R}: a \in I)$ is the formal algebraic expression $f=\bigoplus_{a \in {I}} \left(c_a \odot x^{\odot a}\right)$. Its Newton polytope $\Newt(f)$ is the convex hull of $I$. The polynomial expression $f$ defines a piece-wise linear function from $\mathbb{R}^m$ to $\mathbb{T}$, which by an abuse of notation we shall also denote by $f$, with
$$f(x) = \max_{a \in {I}} \left(c_a + \sum_{i=1}^m a_i x_i\right) \mbox{ for all } x \in \R^m,$$
where we regard $f(x)=-\infty$ for all $x \in \mathbb{R}^m$ if $I= \emptyset$.
Different polynomial expressions may define the same piece-wise linear function \cite[\S 1]{maclagan2015introduction}. In general, it is difficult to check if two tropical polynomials define the same function \cite{lin2017linear}. When $f$ and $g$ each have trivial coefficients (meaning $c_a=0$ for all $a \in I$ and $d_a = 0$ for all $a \in I'$), the following lemma reduces this task to a convex hull computation \cite[\S 1]{maclagan2015introduction}.
\begin{lemma}\label{lem:newton}
Two formal tropical polynomials $f,g$ each having trivial coefficients define the same function if and only if $\Newt(f) = \Newt(g)$.
\end{lemma}
For each formal tropical polynomial $f$ with trivial coefficients there is a canonical reduced expression $[f]$ formed by taking the sum of monomials corresponding to the vertices of $\Newt(f)$ defining the same function; we shall write simply $[f]=[g]$ to denote that two formal tropical polynomials $f$ and $g$ with trivial coefficients define the same function.

Fix an alphabet $\Sigma$ and $\ell \in \mathbb{N}$ and consider a word $w \in \Sigma^\ell$. For each $s \in \Sigma$ let $x(s)$ be the upper triangular matrix whose $(i,j)$-th entry for $i \geq j$ is the variable $x(s,i,j)$ and associate to $w$ the matrix $x(w):= x(w_1) \odot \dots \odot x(w_\ell)$ obtained by multiplying out the corresponding matrices tropically. Note that each entry of $x(w)$ on or above the main diagonal is a tropical polynomial with trivial coefficients. It then follows immediately from the definitions that for $w,v \in \Sigma^+$, 
\begin{equation}
\label{cor:simple}
w \sim_n v \iff [x(w)_{jk}] = [x(v)_{jk}] \mbox{ for each } 1 \leq  j \leq  k \leq n.
\end{equation}
In this formulation, verifying semigroup identities in $\ut{n}$ is equivalent to verifying functional equality of ${n+1 \choose 2}$ pairs of tropical polynomials in $|\Sigma|{n+1 \choose 2}$ variables. This problem has two difficulties. First, obtaining these polynomials through symbolic matrix multiplication is costly. Second, to apply Lemma~\ref{lem:newton}, one must compute the corresponding Newton polytopes, which live in high dimension due to the large number of variables. 

We now derive an alternative characterisation, Theorem \ref{thm:reduce_checks}. It builds on  \cite[Theorem 5.2]{DJK17}, which utilises the well-known connection between tropical matrix multiplication and optimal paths in a weighted directed graph on $n$ nodes \cite{butkovivc2010max,maclagan2015introduction}. The characterisation of Theorem \ref{thm:reduce_checks} also translates $w \sim_n v$ into functional equality of a system of tropical polynomials with trivial coefficients. Though the system is larger than that in  \eqref{cor:simple}, there are fewer variables, and the monomials involved capture combinatorial properties of the two words, meaning that the Newton polytopes are lower dimension and can be constructed directly from the two words. This leads to significant gain in computation and understanding. 

Introduce variables $x(s,i)$ for each letter $s \in \Sigma$ and each $i \in \{1,\ldots, n-1\}$.
Fix a word $w \in \Sigma^+$. Let $d \in \mathbb{N}$ and for each word $u \in \Sigma^{d}$, define the formal tropical polynomial $g_u^w$ associated with $w$ as follows
\begin{equation}\label{eqn:gu.w}
g_{u}^{w}:= \bigoplus_{\pi \in {I}} \bigodot_{s\in \Sigma}\bigodot_{k=1}^{|u|} x(s, k)^{N_s^w(\pi_{k-1}, \pi_{k})}, 
\end{equation}
in variables $x(s,k), s \in \Sigma, 1 \leq k \leq |u|$, where $I$ is the set
\begin{equation}\label{eqn:index.i}
\{(\pi_0, \ldots, \pi_{|u|+1}): 0=\pi_0 < \cdots < \pi_{|u|+1}=|w|+1,  \;  w_{\pi_k}=u_k, 1\leq k \leq |u|\},
\end{equation}
and $N_s^w(\pi_{k-1}, \pi_{k})$ denotes the number of occurrences of $s$ lying strictly between $w_{\pi_{k-1}}$ and $w_{\pi_{k}}$. Recall that $u$ is said to be a scattered subword of $w$ if $w$ has a factorisation of the form
$$w=i(1)\; u_1 \;i(2) \cdots i(d)\; u_d \;i(d+1)$$
 where $i(j) \in \Sigma^*$. The set $I$ thus records the positions of the letters $u_1, \ldots, u_d$ within $w$ for each such factorisation, whilst the exponents $N_s^w(\pi_{k-1}, \pi_{k})$ for $s\in \Sigma$ record the content of the factors $i(k)$ for $k=1, \ldots d$. Notice that, by definition, $g_u^w(x) = -\infty$  for all $x \in \mathbb{R}^{|\Sigma|d}$ if and only if $I = \emptyset$.

\begin{definition}[Signature of a word]
Let $w \in \Sigma^+$. For each $d \in \mathbb{N}$, the {$d$-signature} of $w$ is the lexicographically ordered sequence of Newton polytopes $(\Newt(g_{u}^w):u \in \Sigma^d)$. The {signature of $w$ in $\ut{n}$} is the sequence of its $d$-signatures for $1 \leq d \leq n-1$. 
\end{definition}

\begin{figure}
\resizebox{1\textwidth}{!}{
\begin{tikzpicture}
   \foreach \a in {0,...,3}{
     \draw[help lines]
     (\a,0,0) -- (\a,0,-3)
     (0,0,-\a) -- (3,0,-\a)
     (0,\a,0) -- (0,\a,-3)
     (0,0,-\a) -- (0,3,-\a)
     (0,\a,-3) -- (3,\a,-3)
     (\a,0,-3) -- (\a,3,-3)
     ;
   }
;
\fill[blue!50] (0,0,0)--(1,1,-1)--(2,1,-1);
\node (A1) at (0, -0.2,0) {\footnotesize$(0,0,0)$};
\node (A2) at (0.8, 1.2,-1) {\footnotesize$(1,1,1)$};
\node (A3) at (2.2, 1.2,-1) {\footnotesize$(2,1,1)$};

   \draw[line width=1pt,blue](0,0,0)-- ++(1,0,0);
   \draw[line width=1pt,blue](1,1,-1)-- ++(2,0,0) ;
   \draw[line width=1pt,red] (1,0,0)-- ++(0,1,0);
   \draw[line width=1pt,red](3,1,-1) -- ++(0,1,0);
   \draw[line width=1pt,red](3,2,-2) -- ++(0,1,0);
   \draw[line width=1pt,green](1,1,0)-- ++(0,0,-1);
   \draw[line width=1pt,green](3,2,-1)-- ++(0,0,-1);
   \draw[line width=1pt,green](3,3,-2)-- ++(0,0,-1);
\end{tikzpicture}
\begin{tikzpicture}
   \foreach \a in {0,...,3}{
     \draw[help lines]
     (\a,0,0) -- (\a,0,-3)
     (0,0,-\a) -- (3,0,-\a)
     (0,\a,0) -- (0,\a,-3)
     (0,0,-\a) -- (0,3,-\a)
     (0,\a,-3) -- (3,\a,-3)
     (\a,0,-3) -- (\a,3,-3)
     ;
   }
;
\fill[green!50] (1,1,0)--(3,2,-1)--(3,3,-2);
\node (C1) at (1.2, 1,0) {\footnotesize$(1,0,0)$};
\node (C2) at (3, 1.8,-1) {\footnotesize$(3,1,1)$};
\node (C3) at (2.8, 3.2,-2) {\footnotesize$(3,2,2)$};
\node (C4) at (0, -0.2,0) {\footnotesize$\;$};

         \draw[line width=1pt,blue](0,0,0)-- ++(1,0,0);
   \draw[line width=1pt,blue](1,1,-1)-- ++(2,0,0) ;
   \draw[line width=1pt,red] (1,0,0)-- ++(0,1,0);
   \draw[line width=1pt,red](3,1,-1) -- ++(0,1,0);
   \draw[line width=1pt,red](3,2,-2) -- ++(0,1,0);
   \draw[line width=1pt,green](1,1,0)-- ++(0,0,-1);
   \draw[line width=1pt,green](3,2,-1)-- ++(0,0,-1);
   \draw[line width=1pt,green](3,3,-2)-- ++(0,0,-1);
\end{tikzpicture}
\begin{tikzpicture}
   \foreach \a in {0,...,3}{
     \draw[help lines]
     (\a,0,0) -- (\a,0,-3)
     (0,0,-\a) -- (3,0,-\a)
     (0,\a,0) -- (0,\a,-3)
     (0,0,-\a) -- (0,3,-\a)
     (0,\a,-3) -- (3,\a,-3)
     (\a,0,-3) -- (\a,3,-3)
     ;
   }
;
\fill[red!50] (1,0,0)--(3,1,-1)--(3,2,-2);
\node (B1) at (1, -0.2,0) {\footnotesize$(1,0,1)$};
\node (B2) at (3, 0.8,-1) {\footnotesize$(3,1,2)$};
\node (B3) at (2.8, 2,-2) {\footnotesize$(3,2,3)$};

   \draw[line width=1pt,blue](0,0,0)-- ++(1,0,0);
   \draw[line width=1pt,blue](1,1,-1)-- ++(2,0,0) ;
   \draw[line width=1pt,red] (1,0,0)-- ++(0,1,0);
   \draw[line width=1pt,red](3,1,-1) -- ++(0,1,0);
   \draw[line width=1pt,red](3,2,-2) -- ++(0,1,0);
   \draw[line width=1pt,green](1,1,0)-- ++(0,0,-1);
   \draw[line width=1pt,green](3,2,-1)-- ++(0,0,-1);
   \draw[line width=1pt,green](3,3,-2)-- ++(0,0,-1);
\end{tikzpicture}
}
\caption{The degree $1$ signature of $w=acbaacbcb$.}
\label{sgntre}
\end{figure}
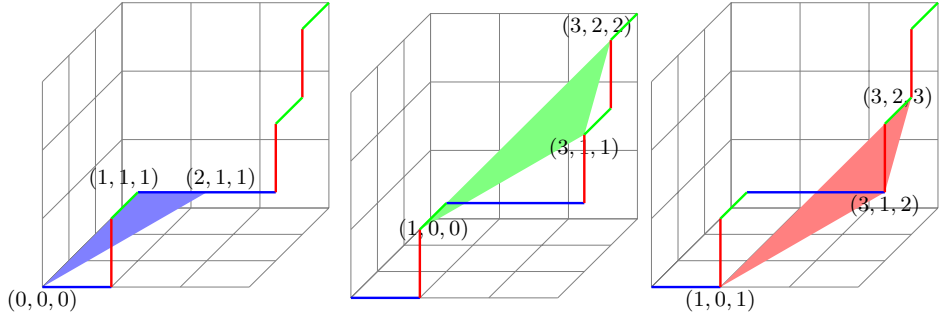

\begin{example}
Let $\Sigma=\{a,b,c\}$ and $w=acbaacbcb$. Then
\begin{eqnarray*}
g_a^w &=& x(a,1)^0x(b,1)^0x(c,1)^0  \oplus x(a,1)^1x(b,1)^1x(c,1)^1 \oplus x(a,1)^2x(b,1)^1 x(c,1)^1 \\
g_b^w &=& x(a,1)^1  x(b,1)^0  x(c,1)^0  \oplus x(a,1)^3 x(b,1)^1 x(c,1)^1  \oplus x(a,1)^3 x(b,1)^2  x(c,1)^2 \\
g_c^w &=& x(a,1)^1  x(b,1)^0 x(c,1)^1  \oplus x(a,1)^3  x(b,1)^1  x(c,1)^2 \oplus x(a,1)^3 x(b,1)^2  x(c,1)^3 .
 \end{eqnarray*} 
The degree $1$ signature of $w$ is $(\Newt(g_{a}^w),\Newt(g_{b}^w), \Newt(g_{c}^w))$, where
\begin{align*}
\Newt(g_{a}^w) &= \conv\{(0,0,0),(1,1,1),(2,1,1)\} \\
\Newt(g_{b}^w) &= \conv\{(1,0,0),(3,1,1),(3,2,2)\} \\
\Newt(g_{c}^w)) &= \conv\{(1,0,1),(3,1,2),(3,2,3)\}.
\end{align*}
These polytopes are illustrated in Figure \ref{sgntre}. 
\end{example}

\begin{theorem}\label{thm:reduce_checks}(c.f. \cite[Theorem 5.2]{DJK17})
Let $w, v \in \Sigma^+$, $n \geq 2$.  Then
\begin{equation}
 \label{eqn:to.prove}
w \sim_n v \Leftrightarrow  [g_{u}^w]=[g_{u}^v]  \mbox{ for all } u \in \Sigma^{d}, 1 \leq d \leq n-1
\end{equation}
or equivalently, $w\sim_n v$ if and only their signatures in $\ut{n}$ are equal.
\end{theorem}

The right hand condition of (\ref{eqn:to.prove}) requires $\sum_{d=1}^{n-1} |\Sigma|^d$ pairs of tropical polynomial functions in $(n-1)|\Sigma|$ variables to be equal. Theorem \ref{thm:reduce_checks} is a simplification of \cite[Theorem 5.2]{DJK17}, which provides a similar condition requiring equality of $\sum_{d=0}^{n-1} \binom{n}{d}|\Sigma|^d$ pairs of tropical polynomials in $n|\Sigma|$ variables. Indeed, our proof proceeds by showing that the two theorems are equivalent. For ease of reference, let us recall \cite[Theorem 5.2]{DJK17}.

\begin{theorem}\cite[Theorem 5.2]{DJK17}\label{thm:djk}
For $0 \leq d \leq n-1$, let $\mathcal{P}^n_d$ denote the set of all increasing sequences of length $d+1$ in $\{1, \ldots, n\}$. Recalling the definition of the set $I$ from \eqref{eqn:index.i}, given $u \in \Sigma^d$, $\rho \in \mathcal{P}^n_d$, and $w \in \Sigma^+$, define
$$f_{u, \rho}^w := \bigoplus_{ \pi \in {I}} \bigodot_{s \in \Sigma} \bigodot_ {k=1}^{|u|+1} x(s, \rho_k)^ {N_s^w(\pi_{k-1}, \pi_{k})}, \mbox{ in variables }x(s,\rho_i), s \in \Sigma, 1 \leq i \leq d+1.$$
For $w,v \in \Sigma^+$, 
$w \sim_n v \Leftrightarrow [f_{u,\rho}^w]=[f_{u, \rho}^v] \, \mbox{ for all } u \in \Sigma^d, \rho \in \mathcal{P}^n_d, 0 \leq d \leq n-1$.
\end{theorem}

\begin{proof}[\textbf{Proof of Theorem \ref{thm:reduce_checks}}]
For each positive integer $d$ let $\tau_d=(1,\ldots, d+1)$. If $w \sim_n v$, then it follows from Theorem \ref{thm:djk} that for each $u \in \Sigma^d, 1 \leq d \leq n-1$ we have in particular $[f_{u, \tau_d}^w]=[f_{u, \tau_d}^v]$. Evaluating at $x(s,d+1)=0$ for all $s \in \Sigma$ then yields $[g_u^w]=[g_u^v]$. For the converse, first notice that if $[g_a^w]=[g_a^v]$, then setting $x(a,1)=1$ and $x(s,1)=0$ for all other $s \in \Sigma$ yields $|w|_a-1=|v|_a-1$. Thus for each $n \geq 2$,  the right hand side of \eqref{eqn:to.prove} implies that $w$ and $v$ have the same content. By \cite[Lemma 2.5]{DJK17} it will therefore suffice to show that the right hand side of \eqref{eqn:to.prove} implies that $\varphi(w)=\varphi(v)$ for all morphisms $\varphi: \Sigma^+ \rightarrow \ut{n}$ such that for all $s \in \Sigma$ and $1 \leq i \leq n-1$, $\varphi(s)_{i,i} = y(s,i) \in \mathbb{R}$ and $\varphi(s)_{n,n}=0$. In this case it is easy to see that  for each $1 \leq j< i \leq n$ we have $\varphi(w)_{i,j} = -\infty =\varphi(v)_{i,j}$ and $\varphi(w)_{j,j} = \bigodot y(s,j)^{|w|_s} = \bigodot y(s,j)^{|v|_s} =\varphi(v)_{j,j}$. It  remains to show that all entries above the diagonal agree;  we shall proceed by induction on $n$. Write $[[g_u^w]]$ to denote evaluation of the function corresponding to the tropical polynomial $g_u^w$ at $x(s,i)=y(s,i)$ for all $s \in \Sigma$ and $1 \leq i \leq n-1$. For $n=2$, it is then easy to see that
\begin{eqnarray*}
\varphi(w)_{1,2} = \bigoplus_{a \in \Sigma} \varphi(a)_{1,2} \odot [[g_a^w]] = \bigoplus_{a \in \Sigma} \varphi(a)_{1,2} \odot [[g_a^w]] =\varphi(v)_{1,2},
\end{eqnarray*}
as required. Now suppose that \eqref{eqn:to.prove} holds for all $2 \leq n \leq k$. If $[g_u^w]=[g_u^v]$ for all $u \in \Sigma^d, 1\leq d \leq k$, then by our inductive hypothesis we have $w \sim_k v$ which by Theorem \ref{thm:djk} yields $[f_{u, \rho}^w]=[f_{u,\rho}^v]$ for all $ u \in \Sigma^d, \rho \in \mathcal{P}^k_d, 0\leq d \leq k-1$. By applying simple changes of variables of the form $x(s,\rho_i) \mapsto x(s, \rho'_i)$, where $\rho$ and $\rho'$ are increasing sequences of the same length over possibly different sets, we may further deduce that $[f_{u, \rho}^w]=[f_{u,\rho}^v]$ for all $ u \in \Sigma^d, \rho \in \mathcal{P}^{k+1}_d, 0\leq d \leq k-1$. Thus if $\varphi: \Sigma^+ \rightarrow \ut{k+1}$ is any morphism and $1 \leq j-i<k$, arguing as in \cite[Lemma 5.1]{DJK17} immediately yields
\begin{eqnarray*}
\varphi(w)_{i,j} &=& \bigoplus_{d=1}^{j-i} \bigoplus_{u \in \Sigma^d} \bigoplus_{\substack{\rho \in \mathcal{P}^{k+1}_d,\\ \rho_0=i, \rho_d=j}} \bigodot_{t=1}^{d} \varphi(u_t)_{\rho_{t-1},\rho_t}\odot  [[f_{u, \rho}^w]]\\
&=& \bigoplus_{d=1}^{j-i} \bigoplus_{u \in \Sigma^d} \bigoplus_{\substack{\rho \in \mathcal{P}^{k+1}_d,\\ \rho_0=i, \rho_d=j}} \bigodot_{t=1}^{d} \varphi(u_t)_{\rho_{t-1},\rho_t}\odot  [[f_{u, \rho}^v]]= \varphi(v)_{i,j}.\\
\end{eqnarray*}
The remaining entries in position $(1,k+1)$ can be similarly expressed and it is easy to see that these will differ only if $[f_{u,\tau_k}^w]$ and $[f_{u, \tau_k}^v$] differ for some word $u$ of length $k$. Since we may assume that $\varphi(s)_{k+1, k+1}=0$ for all $s \in \Sigma$, it follows that a difference occurs in position $(1, k+1)$ only if  the functions corresponding to $f_{u,\tau_k}^w$ and $f_{u, \tau_k}^v$ differ when evaluated at $x(s,k+1)=0$ for all $s \in \Sigma$. But it is straightforward to check that such an evaluation yields the functions corresponding to polynomials $g_u^w$ and $g_u^v$ respectively, which agree by assumption. This completes the proof of \eqref{eqn:to.prove}. The second statement now follows by application of Lemma \ref{lem:newton}, since each polynomial $g_u^w$ has trivial coefficients.
\end{proof}

\section{Computing support sets}
\label{sec:support}
\subsection{Staircase paths, lattice polytopes and degree $1$ supports}
We begin by showing that the degree-one signature of a word $w$ can be easily read off from an associated staircase path.  From Theorem \ref{thm:reduce_checks}, this gives a simple way to construct and verify identities in~$\ut{2}$. Throughout this section, let $\Sigma = \{a_1, \ldots, a_m\}$ and to simplify notation let us write $g_{i}^w$ for the formal tropical polynomial obtained from $g_{a_i}^w$ by the variable substitution $x(a_k,1) \mapsto x_k$. Let $\trianglelefteq$ denote the partial order on $\mathbb{N}_{\geq 0}^m$ given by $p \trianglelefteq q$ if $p_i \leq q_i$ for all $i=1, \ldots, m$.  Recall that for $w \in \Sigma^+$, we write $c(w)=(|w|_{a_1}, \ldots, |w|_{a_m}) \in \mathbb{N}_{\geq 0}^m$  to denote the content of $w$. For $c \in \mathbb{N}_{\geq 0}^m$ let $W(c) = \{w: c(w)=c\}$ be the set of words with prescribed content $c$.

\begin{definition}\label{defn:gamma}
The path $\gamma^w$ of a word $w \in W(\ell_1,\dots,\ell_m)$ is the set of lattice points of the staircase walk from $(0,\dots,0) \in \N_{\geq 0}^m$ to $(\ell_1,\dots,\ell_m) \in \N_{\geq 0}^m$, formed by reading the word $w$ from left to right and moving by the $i$th unit vector $\mathbf{e}_i:=(0, \ldots, 0, 1, 0 \ldots 0)$ when one reads the letter $a_i$, for $i = 1, \dots, m$. For each $i = 1, \dots, m$, the $a_i$-height of $\gamma^w$ is the set
\begin{equation}\label{eqn:gamma.wa}
\gamma^{w,a_i} = \{\gamma_j^{w, a_i}: 0 \leq j \leq \ell_i-1\}, \mbox{ where }\;\; \gamma_j^{w, a_i} := {\rm max}\{ p \in \gamma^w: p_i = j\}.
\end{equation}
\end{definition}

\begin{lemma}\label{lem:support}
For $i = 1, \dots, m$, the support of $g_{i}^w$ is $\gamma^{w,a_i} \subseteq  \mathbb{N}_{\geq 0}^m$.
\end{lemma}
\begin{proof}
Since $\gamma^w$ is a finite increasing sequence of points in $\mathbb{N}_{\geq 0}^m$ with respect to the order $\trianglelefteq$, every nonempty subset of $\gamma^w$ has a maximal element with respect to this order. Thus, the $a_i$-height of $\gamma^w$ is well-defined. For $w \in \Sigma^\ell$, let $w(j) = w_1 \cdots w_j$ denote the length $j$ prefix of $w$. By definition,
\begin{eqnarray*}
g_{i}^w &=& \bigoplus_{h: w_h=a_i} \bigodot_{k=1}^m x_k^{c(w(h-1))_k},\\
\end{eqnarray*}
where $c(w(h-1)) \in \mathbb{N}_{\geq 0}^m$ records the content of the length $h-1$ prefix of $w$.
Suppose that $a_i$ occur in positions $h_0, \ldots, h_{|w|_{a_i} -1}$ of the word $w$. Then the support of $g_i^w$ is $\{c(w(h_j-1)): 0 \leq j \leq |w|_{a_i} -1 \}$. Each such point clearly lies on the path $\gamma^w$ (since it records the content of a  prefix of $w$) and by definition the $i$th coordinate of $c(w(h_j-1))$ is equal to $j$ (since this is the number of occurrences of $a_i$ before the $(j+1)$th occurrence of $a_i$). Moreover, since $w=w(h_j-1)a_i w'$ for some $w' \in\Sigma^*$ it follows that each point $p \in \gamma^w$ that strictly exceeds $c(w(h_j-1))$ in some coordinate is the content of a prefix of length greater than or equal to $h_j$ and hence $p_i > j$. Thus $c(w(h_j -1)) = {\rm max}\{p \in \gamma^w : p_i=j\}$, as required.
\end{proof}

\begin{corollary}
\label{polytopes}
Let $w,v \in \Sigma^+$. Then $w \sim_2 v$ if and only if $\conv(\gamma^{w,a_i})=\conv(\gamma^{v,a_i})$ for all $i=1, \ldots, m$.
\end{corollary}

\begin{example}
\label{adjan}
Figure \ref{fig:adjan.5} illustrates two minimal length identities of the bicyclic monoid \cite{A66}. The left plot shows the paths $\gamma^w$ (black) and $\gamma^v$ (red) in the identity
$$w:=abba \; ab \; abba \; \sim_2 \;  abba \; ba \; abba =:v.$$ 
It is easy to see that $\conv(\gamma^{w,a})=\conv(\gamma^{v,a})$ and $\conv(\gamma^{w,b})=\conv(\gamma^{v,b})$ (the polygons shown in blue and green). Similarly, the right plot shows the paths $\gamma^w$ (black) and $\gamma^v$ (red) of the words in the identity  
$$w=abba \; ab \; baab\; \sim_2 \;abba \; ba \; baab=: v.$$
\end{example}
\begin{figure}[h!]
\includegraphics[width=0.48\textwidth]{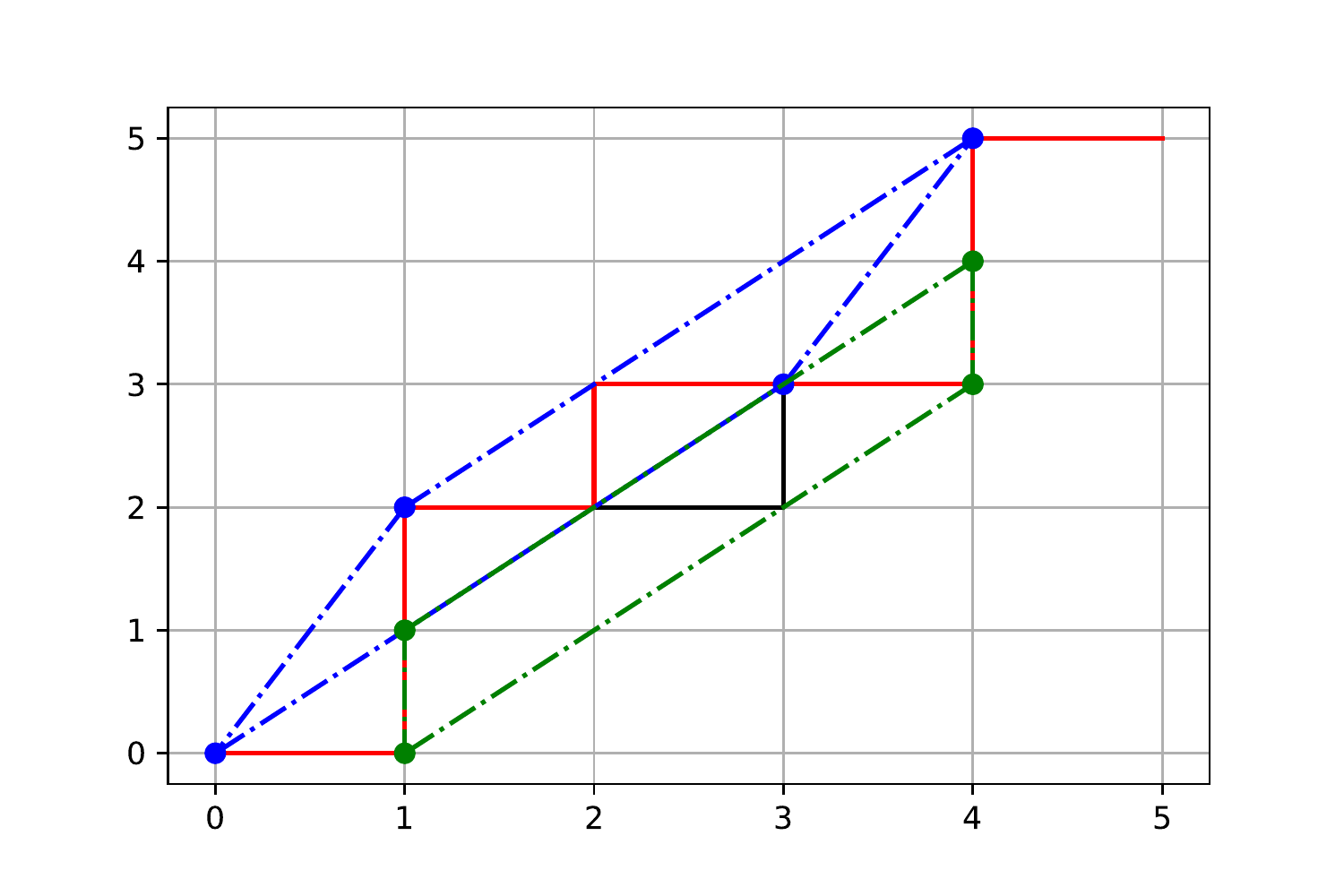} \, 
\includegraphics[width=0.48\textwidth]{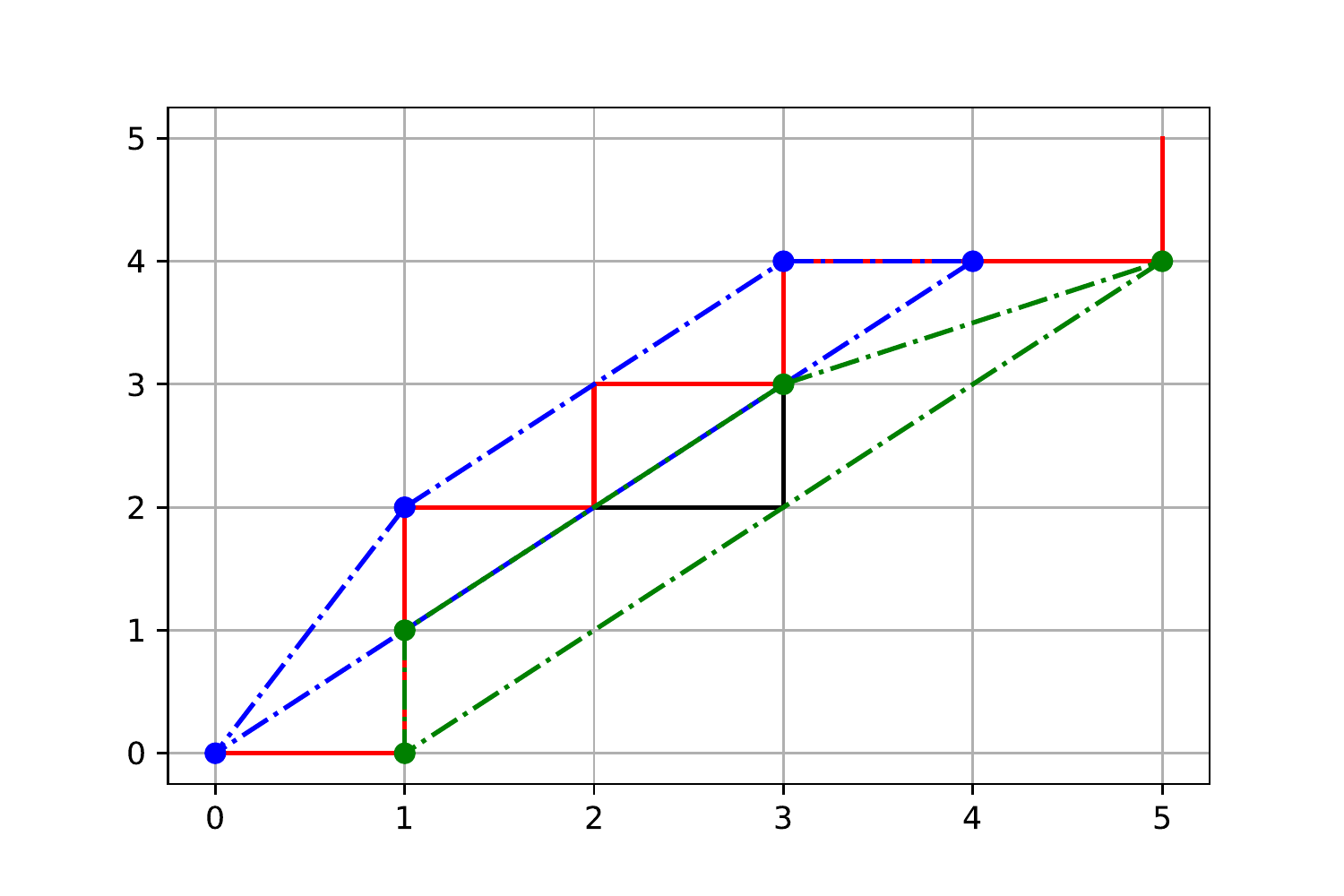}
\caption{Visualisation of the identities of Example \ref{adjan}.}
\label{fig:adjan.5}
\end{figure}
\begin{example}
The degree-$1$ signature of $w=acbaacbcb$, along with its path is illustrated in Figure \ref{sgntre}. Note that each point of $\gamma^w$, save the final one, is a vertex of one of the polytopes in its signature. So the path, and hence the word, is uniquely determined from this signature, showing that $w$ is an isoterm for $\ut{2}$.
\end{example}

\subsection{Degree $n$ support sets}
Fix an alphabet $\Sigma = \{a_1, \ldots, a_m\}$ and a word $w \in \Sigma^+$. For each scattered subword $u$ of $w$ with $|u|=d$, let $\gamma^{u}\subseteq \mathbb{N}_{\geq 0}^{md}$ denote the nonempty support set of the polynomial $g_u^w$. Proposition \ref{prop:transformation} shows how these support sets can be recursively computed from the staircase path of $w$. We begin with an example to illustrate the idea. 

\begin{example}
Let $\Sigma = \{a,b\}$ and consider the support set $\gamma^{ab}$ of the polynomial $g^w_{ab}$ in variables $x(a,1), x(b,1), x(a,2),x(b,2)$. To compute this, we must consider each instance of $ab$ as a scattered subword of $w$. Suppose that for some $0 \leq i \leq |w|_a -1$ and $0 \leq j \leq |w|_b -1$, the $(i+1)$th $a$ occurs before the $(j+1)$th $b$. Then \emph{before} the $(i+1)$th $a$ there are $\gamma^{w,a}_i$ many $a$'s and $b$'s. \emph{Between} the $(i+1)$th $a$ and the $(j+1)$th $b$ there are $\gamma^{w,b}_j - \gamma^{w,a}_i- \mathbf{e}_1$ many $a$'s and $b$'s. Thus  if $(i, i', j', j) \in \gamma^{ab}$ then we have $(i,i') = \gamma^{w,a}_i$, and $(j', j) = \gamma^{w,b}_j - \gamma^{w,a}_i- \mathbf{e}_1$. It follows easily from this observation that
$$\gamma^{ab} \simeq (\gamma^{w,a} \times \gamma^{w,b}) \cap C_a,$$
where the inequalities specifying the polyhedron
$$C_a = \{(i,i',j',j)  \in \R_{\geq 0}^4: i < j', \mbox{ and }i' \leq j \}$$ ensure that the $(j+1)$th $b$ is occurring \emph{after} the $(i+1)$th $a$.
\end{example}

\begin{proposition}
\label{prop:transformation}
Let $w \in \Sigma^+$ and let $u$ be a scattered subword of $w$ of length $d$. Then for each $a_j \in \Sigma$
\begin{eqnarray*}
&&\gamma^{ua_j} \simeq  (\gamma^{u} \times \gamma^{a_j}) \cap C_u,\\
\mbox{ where  }&&C_u =\{y \in \mathbb{R}^{m(d+1)}: \sum_{k=0}^{d-1} y_{mk+r} + |u|_{a_r} \leq y_{md+r}, 1 \leq r \leq m\}.
\end{eqnarray*}
\end{proposition}
\begin{proof}
Consider the invertible affine linear map $\Pi: \mathbb{R}^{m(d+1)} \rightarrow \mathbb{R}^{m(d+1)}$ defined by
$$\Pi(p)= \left( 
\begin{array}{c|c}
I_{md} & 0 \\
\hline
\begin{array}{c|c|c}
I_m&\cdots &I_m\\
\end{array}& I_m\\
\end{array}
\right)p + \sum_{r=1}^m |u|_{a_r} \mathbf{e}_{md+r},$$
where $\mathbf{e}_i$ denotes the $i$th unit vector of $\mathbb{R}^{m(d+1)}$. We claim that the image of $\gamma^{ ua_j}$ under $\Pi$ is $(\gamma^{u} \times \gamma^{ a_j}) \,\cap\, C_u$. 
Let $p \in \gamma^{ ua_{j}}$. Then, by definition, there is a factorisation of $w$ of the form:
$$w=i(0)\;u_1\;i(1)\;u_2\; \cdots \;i(d-1)\;u_d \; i(d)\;  a_j \; w',$$
in which $i(0), \ldots, i(d),w' \in \Sigma^*$ and 
$$p_{km+r} = |i(k)|_{a_r} \mbox{ for } 0 \leq k \leq d,\;  1 \leq r \leq m.$$
Then 
$$\Pi(p) = (p_1, \ldots, p_{md}, \; \sum_{k=0}^{d} p_{km+1} +|u|_{a_1}, \ldots, \; \sum_{k=0}^{d} p_{km+m} +|u|_{a_m}).$$ 
Since the first $md$ coordinates describe the first $d$ intermediary factors in our factorisation of $w$, it is immediate that restricting to these coordinates yields a point of $\gamma^{u}$. Noting that the final $m$ coordinates describe the content of the prefix $v:=i(0)\;u_1\;i(1)\;u_2\; \cdots \;i(d-1)\;u_d \; i(d)$ of $w=va_jw'$, it follows that restricting to these coordinates gives a point of $\gamma^{a_j}$.
Finally, $0 \leq p_{md+r}$ for each $r=1, \ldots m$, so $\Pi(p)$ satisfies the inequalities specified by $C_u$. This shows that $\Pi(\gamma^{ ua_j}) \subseteq (\gamma^{u} \times \gamma^{ a_j}) \cap C_u$. For the reverse inclusion, suppose that $y \in (\gamma^{u} \times \gamma^{a_j}) \cap C_u$. By definition of $\gamma^{u}$, there is a factorisation of $w$ of the form:
$$w=i(0)\;u_1\;i(1)\;u_2\; \cdots \;i(d-1)\;u_d \;w',$$
in which $i(0), \ldots, i(d-1),w'\in \Sigma^*$ and
$$y_{km+r} = |i(k)|_{a_r} \mbox{ for } 0\leq k \leq d-1, 1 \leq r \leq m.$$
Similarly, by the definition of $\gamma^{a_j}$, there is another a factorisation of $w$ of the form $w=va_jw''$ in which $v,w'' \in \Sigma^*$ and
$$y_{md+r} = |v|_{a_r} \mbox{ for } 1 \leq r \leq m.$$
The inequalities of $C_u$ then allow us to deduce that $i(0)\;u_1\;i(1)\;u_2\; \cdots \;i(d-1)\;u_d$ must be a prefix of $v$, so that
$$w=i(0)\;u_1\;i(1)\;u_2\; \cdots \;i(d-1)\;u_d \; i(d)\;  a_j \; w'',$$
for some $i(d) \in \Sigma^*$. Taking $p$ to be the point of $\gamma^{ua_j}$ corresponding to this factorisation then yields $\Pi(p)=y$, so $\Pi(\gamma^{ ua_j}) \supseteq (\gamma^{u} \times \gamma^{ a_j}) \cap C_u$. Thus we have equality between these two sets as claimed.
\end{proof}

\begin{remark}
It follows from the previous result that every support set can be built from the degree $1$ supports by taking set products, intersecting with appropriate polyhedra and applying affine linear transformations. Notice however that one cannot build the degree $d$ signature of a word directly from the degree $1$ signature \emph{polytopes} in this way, since although taking convex hulls commutes with affine transformations and set products, it need not commute with taking intersections and so
it may happen that
$${\rm conv}(\gamma^{u} \times \gamma^{a} \cap C_u) \neq {\rm conv}(\gamma^{u} \times \gamma^{a}) \cap C_u.$$
Therefore, Proposition \ref{prop:transformation} gives a clear geometric interpretation as to why $w \sim_n v$ need not imply $w \sim_{n+1} v$.
\end{remark}

\subsection{Properties of $\ut{n}$ identities}
 Reversal of words yields an involution on $\sim_n$ equivalence classes.
\begin{lemma}
Let $w, v \in \Sigma^+$, and let ${\rm rev}(w), {\rm rev}(v)$ denote the words obtained from $w$ and $v$ by reading from right-to-left. Then $w \sim_n v$ if and only if ${\rm rev}(w) \sim_n {\rm rev}(v)$.
\end{lemma}

\begin{proof}
Let $w \sim_n v$ and consider the involution $R: \ut{n} \rightarrow \ut{n}$ obtained by reflecting along the anti-diagonal. It is straightforward to check that $R$ is an anti-automorphism of $\ut{n}$, meaning that for all $A, B \in \ut{n}$ one has $R(AB) = R(B)R(A)$. Let $\varphi: \Sigma^+ \rightarrow \ut{n}$ be a morphism. Consider the morphism $\psi: \Sigma^+ \rightarrow \ut{n}$ constructed from $\varphi$ by defining $\psi(a) = R(\varphi(a))$ for all $a \in \Sigma$. Since $R$ is an involution, we have
\begin{eqnarray*}
\varphi(w_{\ell}\cdots w_1) &=& \varphi(w_{\ell})\cdots \varphi(w_1) = R(\psi(w_{\ell})) \cdots R(\psi(w_1)) = R(\psi(w_1) \cdots \psi(w_{\ell}) )\\
\varphi(v_{\ell}\cdots v_1) &=& \varphi(v_{\ell})\cdots\varphi(v_1) = R(\psi(v_{\ell})) \cdots R(\psi(v_1)) = R(\psi(v_1)\cdots\psi(v_{\ell})).
\end{eqnarray*}
Since $\psi$ is a morphism and $w \sim_n v$, the two expressions are equal.
\end{proof}

\begin{remark}
\label{rev}
In the case where $n=2$, the previous lemma can also be read off from the degree $1$ signature. For each $i = 1, \dots, m$, define the map $T_i^w: \mathbb{R}^m \rightarrow \mathbb{R}^m$ via $p \mapsto  - p + c(w) -\mathbf{e}_i$. It is straightforward to check that $T(\gamma^{w,a_i}) = \gamma^{{\rm rev}(w), a_i}$, and that $T$ is affine linear and invertible. Thus $\conv(\gamma^{w,a_i}) = \conv(\gamma^{v,a_i})$ if and only if $\conv(\gamma^{{\rm rev}(w),a_i}) = \conv(\gamma^{{\rm rev}(v),a_i})$, and the conclusion follows. 
\end{remark}

The following lemma shows how some properties of identities for $\ut{n}$ manifest themselves in terms of geometric properties of the degree $1$ signature of the words. 

\begin{lemma}
\label{props}
Let $i \in \{1,\ldots, m\}$, $n\geq2$, and $w, v \in \Sigma^+$.
\begin{itemize}
\item[(i)] The minimal and maximal elements of $\gamma^{w,a_i}$ are vertices of $\conv(\gamma^{w,a_i})$. In particular, $w \sim_n v$ implies that $c(w)=c(v)$.

\item[(ii)] If $w = w' a_i^k w''$ where $w'$  does not contain $a_i$ and $w''$ does not start with $a_i$, then  $\conv(\gamma^{w,a_i})$ has vertices $c(w')$ and $c(w')+(k-1)\mathbf{e}_i$. In particular, $w \sim_n v$ implies $v=v'a_i^kv''$, where $c(v')=c(w')$ and $v''$ does not start with $a_i$.

\item[(iii)] If $w = w' a_i^k w''$ where $w''$  does not contain $a_i$ and $w'$ does not end with $a_i$, then $\conv(\gamma^{w,a_i})$ has vertices $c(w')$ and $c(w')+(k-1)\mathbf{e}_i$. In particular, $w \sim_n v$ implies $v=v'a_i^kv''$, where $c(v'')=c(w'')$ and $v'$ does not end with $a_i$.

\item[(iv)] If $w=a_{i_1}^{k_1}\cdots a_{i_t}^{k_t}w''$ where the indices $i_1, \ldots, i_t$ are distinct elements  of $\{1,\ldots, m\}$, $k_1, \ldots, k_t, \in \mathbb{N}$, and $w'' \in \Sigma^* \setminus a_{i_t}\Sigma^*$, then each $v$ in the $\sim_n$-class of $w$ has the form  $v=a_{i_1}^{k_1}\cdots a_{i_t}^{k_t}v''$ where $v'' \in \Sigma^*\setminus a_{i_t}\Sigma^*$ with $c(v'')=c(w'')$ and the first letter of $v''$ is equal to the first letter of $w''$.
\end{itemize}
\end{lemma}

\begin{proof}
Noting that $w \sim_n v$ implies $w \sim_2 v$, it suffices to prove each implication in the case $n=2$.

(i) Since $\gamma^{w,a_i}$ is a finite increasing sequence of points in $\mathbb{N}_{\geq 0}^m$ with respect to the natural partial order $\trianglelefteq$, it is immediate that the minimal and maximal elements of this set must be vertices of the convex hull $\conv(\gamma^{w,a_i})$. Suppose then that $w \sim_2 v$. Corollary \ref{polytopes} then implies that the maximal element of $\cup_{j=1}^{m} \gamma^{w, a_j}$ must be equal to the maximal element of $\cup_{j=1}^{m} \gamma^{v, a_j}$, and that this common point lies in $\gamma^{w,a_i} \cap \gamma^{v,a_i}$ for some $i$. Since the maximal element of $\gamma^{w, a_i}$ is $c(w)-\mathbf{e}_i$, the result follows. 

(ii) Noting that $\gamma^{w, a_i}$ is contained in the affine cone $c(w') + \mathbb{R}_{\geq 0}^m$, we see that the segment of points $c(w'a_i^j)$ with $0 \leq j \leq k-1$ are the unique elements of $\gamma^{w, a_i}$ lying on the extremal ray $c(w') + \mathbb{R}_{\geq 0}\mathbf{e}_{i}$. Thus the two extremal points of this segment, namely $c(w')$ and $c(w'a^{k-1})$, must be vertices of $\conv(\gamma^{w, a_i})$. If $w \sim_2 v$, then these two points must also be vertices of  $\conv(\gamma^{v, a_i})$, from which we deduce that $v=v' a_i^kv''$ where $c(v')=c(w')$ and the first letter of $v''$ is not equal to $a_{i}$.

(iii) If $w''$ does not contain $a_i$, then the polytope $\conv(\gamma_i^{{\rm rev}(w)})$ has vertices $c({\rm rev}(w'')), c({\rm rev}(w'')) + (k-1)\mathbf{e}_i$ and applying the transformation from Remark \ref{rev} then yields that the polytope $\conv(\gamma_i^{w})$ has vertices: $c(w'), c(w') + (k-1)\mathbf{e}_i$.

(iv) By repeated application of part (ii) one may deduce that $v=w'v''$ where  $w'= a_{i_1}^{k_1}\cdots a_{i_t}^{k_t}$ and where the first letter of $v''$ is not equal to $a_{i_t}$. Now let  $a_i$ (respectively, $a_j$) denote the first letter of $w''$ (respectively, $v''$). It remains to show that $j=i$. Suppose first that $j \not \in \{i_1, \ldots, i_{t-1}\}$. Since $j \neq i_t$, reasoning as above yields that $c(w')$ must be a \emph{vertex} of $\conv(\gamma^{v, a_j}) = \conv(\gamma^{w,a_j})$, and so in particular $c(w') \in \gamma^{w,a_j}$. Since the sets $\gamma^{w,a_1}, \ldots, \gamma^{w,a_m}$ are disjoint and $c(w') \in \gamma^{w,a_i}$ we obtain $j=i$. Likewise, if $i \not \in \{i_1, \ldots, i_{t-1}\}$, we obtain $i=j$.

Now suppose that $i,j \in \{i_1,\ldots, i_{t-1}\}$ and suppose for contradiction that $j \neq i$. Without loss of generality we may assume that $i=i_r$ and $j=i_s$ where $1 \leq r<s<t$.  Since $\ut{n}$ contains an identity element, it is clear that if $w \sim_n v$, then $(w\setminus \Delta) \sim_n (v \setminus \Delta)$, where $(w\setminus \Delta)$ and $(v \setminus\Delta)$ are the words obtained from $w$ and $v$ by deleting all occurrences of the letters from any proper subset $\Delta$ of of $\Sigma$. Taking $\Delta = \Sigma \setminus \{i,j\}$ we find that $(w \setminus \Delta) = a_i^{k_r} a_j^{k_s} a_i^k (w'' \setminus \Delta)$ and $(v \setminus \Delta) = a_i^{k_r} a_j^{k_s} a_j^m (v'' \setminus \Delta)$. Applying the reasoning of the first  paragraph to these two words then yields a contradiction.
\end{proof}

Say that $a_i^k$ is a block of $w \in \Sigma^+$ if $w=w'a_i^kw''$ where $w'$ does not end with $a_i$ and $w''$ does not begin with $a_i$. Lemma \ref{props} then tells us that any word with very few blocks of each letter must be an isoterm.
\begin{corollary}
Let $w \in \Sigma^+$ be a word in which each letter $a_i \in \Sigma$ occurs in at most two distinct blocks. Then $w$ is an isoterm for $\ut{2}$.
\end{corollary}

\begin{proof}
 By Lemma \ref{props} (ii), the first block of each letter fixes a pair of points through which every path $\gamma^v$ with $v \sim_2 w$ must pass.  Likewise,  by Lemma \ref{props} (iii), the last block of each letter fixes a pair of points through which every path $\gamma^v$ with $v \sim_2 w$ must pass. If each letter occurs in at most two distinct blocks, then it follows that each block of $w$ is either the first block of that letter or the last block of that letter (possibly both), and it is straightforward to verify that the points described above uniquely determine the word.
\end{proof}

\begin{corollary}
Let $|\Sigma| = 2$. If the total number of blocks of $w \in \Sigma^+$ (counted with multiplicity) is no more than five, then $w$ is an isoterm for $\ut{2}$.
\end{corollary}
\begin{proof}
If $w$ is a word on two letters with fewer than five blocks, then it is clear that each block is either the first or last block of a letter, and the reasoning of the previous corollary applies. Suppose then that $w$ has five blocks, say $w=a^h b^i a^j b^k a^l$, and that $w \sim_2 v$.  Parts (ii) and (iii) of Lemma \ref{props} yield that $v=a^hb^i v''$ and $v=v'b^k a^l$ for some $v', v'' \in \Sigma^*$. Since $v$ and $w$ are required to have the same content we see that $v=w$.
\end{proof}
More generally, one can use similar  reasoning to show that any word over an $m$-letter alphabet containing no more than $m+3$ blocks must be an isoterm. For the sake of brevity we omit the proof.

Say that an identity system $E\subset W(c) \times W(c)$ is left $1$-hereditary if for each pair $(v,w) \in E$,  whenever $v=v'zv''$ and $w=w'zw''$ where $v'$ and $w'$ do not contain any occurrences of $z \in \Sigma$, then $(v',w') \in E$. Dually, say that $E$ is right $1$-hereditary if for each pair $(v,w) \in E$, we have that whenever $v=v'zv''$ and $w=w'zw''$ where $v''$ and $w''$ do not contain any occurrences of $z \in \Sigma$, then $(v'',w'') \in E$. The equational theory of the bicyclic monoid is known to be both left and right $1$-hereditary \cite{Shn89,P06}. Since $\ut{2}$ identities are exactly the identities of the bicyclic monoid \cite[Theorem 4.1]{DJK17}, we obtain a geometric proof of this statement by consideration of the degree $1$ signatures of the words. The following easy lemma (depicted in Figure \ref{fig:outofbox}) captures the essence of this property in terms of the polytopes involved.

\begin{lemma}
\label{outofbox}
Let $c\in \mathbb{R}_{\geq 0}^m$, with $c_j \neq 0$ for some $j$, and let $B_{\hat{c}} = \{ p\in \mathbb{R}_{\geq 0}^m : p_j=0, p_i \leq c_j \mbox{ for all }i \neq j\}$.  If $P$ and $Q$ are finite sets with $P \subseteq B_{\hat{c}}$ and $Q \subseteq c +\mathbb{R}_{\geq 0}^m$, then every vertex of $\conv(P)$ is a vertex of $\conv(P \cup Q)$.
\end{lemma}

\begin{proof}
It is easily verified that $\conv(P \cup Q ) \cap B_{\hat{c}} = \conv(P)$. Suppose $p$ is a vertex of $\conv(P)$, but not of $\conv(P \cup Q)$. Then $p$ lies on a line segment strictly between $q$ and $r$ for some $q,r \in \conv(P \cup Q)$. Taking one of $q,r$ outside the box $B_{\hat{c}}$ contradicts that $p$ lies in this box, whilst taking both $q,r$ in $B_{\hat{c}}$, and hence in $\conv(P)$, contradicts that $p$ is a vertex of $P$.   
\end{proof}
\begin{figure}
\captionsetup{width=1\linewidth}
\resizebox{0.5\textwidth}{!}{
\begin{tikzpicture}
   \foreach \a in {0,...,5}{
     \draw[help lines]
     (0,\a,0) -- (0,\a,-5)
     (0,0,-\a) -- (0,5,-\a);
   }

\draw[help lines]  (3,0,0) -- (0,0,0);
\draw[help lines]  (0,5,0) -- (0,0,0);
\draw[help lines]  (0,0,-5) -- (0,0,0);

\path (0,0,0) node{\textbullet};
\path (0,0,-4) node{\textbullet};
\path (0,1,-2) node{\textbullet};
\path (0,3,0) node{\textbullet};
\path (0,4,-3) node{\textbullet};
\path (0,3,-4) node{\textbullet};
\path (0,2,-3) node{\textbullet};
\path (0,3,-2) node{\textbullet};
\path (3,5,-5) node{\textbullet};
\path (2.7,5.3,-5) node{\huge$c$};
\path (-0.3,0.2,0) node{\huge$0$};
\path (7,6.8,-7) node{\huge$\conv(Q)$};
\path (0,2,-2) node{\huge$\conv(P)$};
\path (-0.5,4,0) node{\huge$B_{\hat{c}}$};
\draw[dashed]  (0,0,-4) -- (0,1,-2);
\draw[dashed]  (0,0,-4) -- (0,3,-4);

\draw[dashed]  (0,1,-2) -- (0,3,0);

\draw[dashed]  (0,3,0) -- (0,4,-3);

\draw[dashed]  (0,4,-3) -- (0,3,-4);

\draw[help lines]  (3,5,-5) -- (3,5,-10);
\draw[help lines]  (3,5,-5) -- (3,10,-5);
\draw[dashed]  (0,5,-5) -- (3,5,-5);
\draw[help lines]  (3,5,-5) -- (8,5,-5);
\draw[dashed]  (3,5,-5) -- (3,0,-5);
\draw[dashed]  (3,0,-5) -- (0,0,-5);
\draw[dashed]  (3,0,0) -- (3,0,-5);
\draw[dashed]  (3,5,-5) -- (3,5,0);
\draw[dashed]  (3,5,0) -- (3,0,0);
\draw[dashed]  (3,5,0) -- (0,5,0);

\path (7,8,-5) node{\textbullet};
\path (5,7,-8) node{\textbullet};
\path (7,8,-9) node{\textbullet};
\path (6,7,-6) node{\textbullet};
\path (6,9,-6) node{\textbullet};
\path (7,7,-10) node{\textbullet};
\draw[dashed]  (7,8,-5) -- (5,7,-8);
\draw[dashed]  (7,8,-5) -- (7,8,-9);
\draw[dashed]  (7,8,-5) -- (6,7,-6);
\draw[dashed]  (7,8,-5) -- (6,9,-6);
\draw[dashed]  (7,8,-5) -- (7,7,-10);
\draw[dashed]  (5,7,-8) -- (7,8,-9);
\draw[dashed]  (5,7,-8) -- (6,7,-6);
\draw[dashed]  (5,7,-8) -- (6,9,-6);
\draw[dashed]  (5,7,-8) -- (7,7,-10);
\draw[dashed]  (7,8,-9) -- (6,7,-6);
\draw[dashed]  (7,8,-9) -- (6,9,-6);
\draw[dashed]  (7,8,-9) -- (7,7,-10);
\draw[dashed]  (6,7,-6) -- (6,9,-6);
\draw[dashed]  (6,7,-6) -- (7,7,-10);
\draw[dashed]  (6,9,-6) -- (7,7,-10);
\end{tikzpicture}
}

\caption{Geometric interpretation of the left $1$-hereditary property}
\label{fig:outofbox}
\end{figure}
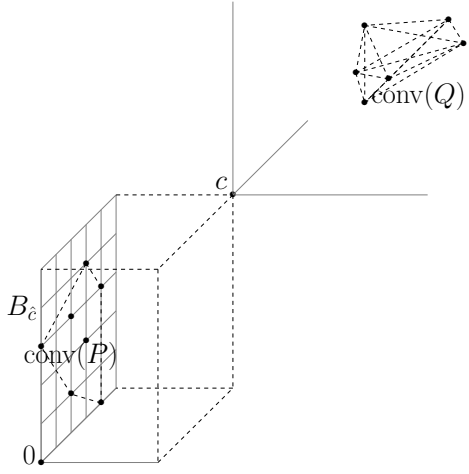

\begin{proposition}\cite[Lemma 2]{Shn89} (c.f. \cite[4.16]{P06})
The equational theory of the bicyclic monoid is both left and right $1$-hereditary.
\end{proposition}

\begin{proof}
Suppose that $v \sim_2 w$, where $v=v'a_jv''$ and $w=w'a_jw''$ and where $v'$ and $w'$ do not contain any occurrences of $a_j \in \Sigma$. Then we have $\conv(\gamma^{w,a_i}) = \conv(\gamma^{w, a_i})$ for all $i$. The fact that the $j$-polytopes agree tells us that $c(w'a_j) = c(v'a_j)$; let $c$ denote this common content. For each $i$ such that $c_i \neq 0$, the points of $\gamma^{w,a_i}$ (respectively, $\gamma^{v, a_i}$) each lie in either the box $B_{\hat{c}}$ or the cone $c + \mathbb{R}_{\geq 0}^m$. The points lying in the box are precisely the points of $\gamma^{w', a_i}$ (respectively, $\gamma^{v', a_i}$). Applying Lemma \ref{outofbox}, first with $P=\gamma^{w', a_i}$ and the set difference $Q=\gamma^{w,a_i}\setminus \gamma^{w', a_i}$, and second with $P=\gamma^{v', a_i}$ and the set difference $Q=\gamma_i^{v, a_i}\setminus \gamma^{v',a_i}$, then yields that $\conv(\gamma^{w', a_i}) = \conv(\gamma^{v',a_i})$. Thus  $v' \sim_2 w'$.
\end{proof}

\section{Identities for the bicyclic monoid on two letter alphabets}
\label{sec:main}
This section collects our main results on $\ut{2}$ identities over a two-letter alphabet. In particular, we prove Theorem \ref{thm:chain}, providing an algorithm to compute the minimal and maximal elements of each equivalence class guaranteed by this theorem (cf. Theorem \ref{thm:min.max}), and give an alternative proof of the minimality of Adjan's identity for the bicyclic monoid. 

Let $\Sigma = \{a,b\}$. The degree-1 signature of a word $w \in W(\ell_a,\ell_b)$ is the pair of polygons $(\Newt(g_a^w), \Newt(g_b^w))$, which we denote as $A(w)$ and $B(w)$, respectively. To further simplify notation, we write $\alpha^w$ for $\gamma^{w,a}$ and $\beta^w$ for $\gamma^{w,b}$ defined in (\ref{eqn:gamma.wa}) and denote the unique element of $\alpha^w$ with first coordinate $0 \leq i \leq \ell_a -1$ by  $(i, \alpha_i^w)$, and the unique element of $\beta^w$ with second coordinate $0 \leq i \leq \ell_b -1$ by $(\beta_i^w,i)$.  The dual of $w$ is the word $\widetilde{w} \in W(\ell_b,\ell_a)$ obtained by swapping each occurrence of the letter $a$ in $w$ by the letter $b$, and vice versa.  Note that $(i,\alpha^{w}_i) = (i,\beta^{\widetilde{w}}_i)$ for all $0 \leq i \leq \ell_a-1$. 
 
There is a one-to-one correspondence between a word and its path; for words on two letters, fixing the length and the $a$-height $\alpha^{w}$ also uniquely determines $w$.
\begin{lemma}\label{lem:w.from.a}
Let $\ell_a, \ell_b \in \mathbb{N}$ and $\alpha: \{0,1\dots,\ell_a-1\} \to \mathbb{N}_{\geq 0}$  a non-decreasing map $i \mapsto \alpha_i$. If $\ell_b \geq \alpha_{\ell_a-1}$, then
\begin{equation}\label{eqn:w.from.a}
w = b^{\alpha_0}\,a\,b^{(\alpha_1-\alpha_0)}\,a\,b^{(\alpha_2-\alpha_1)}\,a\, \cdots b^{(\alpha_{\ell_a-1}-\alpha_{\ell_a-2})}\,a\,b^{(\ell_b-\alpha_{\ell_a-1})}
\end{equation}
is the unique word in $W(\ell_a,\ell_b)$ such that $\alpha^{w} = \{(i, \alpha_i): 0 \leq i \leq \ell_a -1\}$.
\end{lemma}

Equation (\ref{eqn:w.from.a}) elucidates the interpretation of $\alpha^{w}_i$ for each $i$.  Applying the same argument to the dual word $\widetilde{w}$, one sees that the word length $\ell$ together with the $b$-height $\beta^w$ also uniquely determines $w$. Lemma \ref{lem:w.from.a} and Corollary \ref{polytopes} therefore converts the study of $\ut{2}$ identities over a two letter alphabet to the study of non-decreasing functions $\alpha: \{0,1\dots,\ell_a-1\} \to \{0,1,\dots,\ell_b\}$ for each fixed pair of non-negative integers $(\ell_a,\ell_b)$.  

Fix $\ell_a,\ell_b \in \mathbb{N}$. For staircase walks $\gamma, \gamma'$ that both start from $(0,0)$ and end at $(\ell_a,\ell_b)$, define $\gamma \preceq \gamma'$ if the walk $\gamma'$ does not go above the walk $\gamma$, and $\gamma \prec \gamma'$ if additionally the two paths are not equal. Define $w \preceq w'$ if $\gamma^w \preceq \gamma^{w'}$. This turns $W(\ell_a,\ell_b)$ into a distributive lattice. The following lemma follows straightforwardly from the definitions.

\begin{lemma}
For $w,v \in W(\ell_a,\ell_b)$, the following are equivalent:
\begin{itemize}
\item $\gamma^w \preceq \gamma^v$.
\item $(k, \alpha^w_k) \trianglelefteq (k,\alpha^v_k)$ for all $k \in \{0, \ldots, \ell_a-1\}$.
\item $(\beta^w_k, k) \trianglerighteq (\beta^v_k, k)$ for all $k \in \{0, \ldots, \ell_b -1\}$.
\end{itemize}
In particular, the $a$-height and $b$-height of $w \wedge v$ and $w \vee v$ are given by
\begin{eqnarray*}
\alpha^{w \wedge v}_i &=& \min(\alpha^{v}_i,\alpha^{w}_i), \quad \beta^{w \wedge v}_i = \max(\beta^{v}_i,\beta^{w}_i), \\
\alpha^{w \vee v}_i &=& \max(\alpha^{v}_i,\alpha^{w}_i), \quad \beta^{w \vee v}_i = \min(\beta^{v}_i,\beta^{w}_i).
\end{eqnarray*}
\end{lemma}

\begin{proof}[\textbf{Proof of Theorem \ref{thm:chain}}]
Let $w,u, v \in W(\ell_a, \ell_b)$ and suppose that $w\sim_2 v$. Thus we may write $A:=A(w)=A(v)$ and $B:=B(w)=B(v)$. 

\noindent(i) Suppose $w \preceq u \preceq v$. By definition, we have  
\begin{eqnarray*}
A&=&\conv{\{(i, \alpha_i^w): 0 \leq i \leq \ell_a - 1\}}=\conv{\{(i, \alpha_i^v): 0 \leq i \leq \ell_a - 1\}}, \mbox{ and }\\
A(u)&=&\conv{\{(i, \alpha_i^u): 0 \leq i \leq \ell_a-1\}}. 
\end{eqnarray*}
Since $w \preceq u \preceq v$, we have $\alpha_i^w \leq \alpha_i^u \leq \alpha_i^v$ for all $i$. Thus  each point $(i, \alpha_i^u)$ lies on a (vertical) line segment contained within $A$, giving $A(u) \subseteq A$. If $p$ is a vertex of $A$, then clearly $p,p+(1,0) \in \gamma^w \cap \gamma^v$, from which it follows that $p, p+(1,0) \in \gamma^u$, giving $A \subseteq A(u)$. Thus $A(u) = A$. Likewise, 
\begin{eqnarray*}
B&=&\conv{\{(\beta_i^w,i): 0 \leq i \leq \ell_b - 1\}}=\conv{\{(\beta_i^v, i): 0 \leq i \leq \ell_b -1\}}, \mbox{ and }\\
B(u)&=&\conv{\{(\beta_i^u, i): 0 \leq i \leq \ell_b-1\}}.
\end{eqnarray*}
Since $w \preceq u \preceq v$, we have $\beta_i^w \geq \beta_i^u \geq \beta_i^v$ for all $i$. Thus each point $(\beta_i^u,i)$ lies on a (horizontal) line segment contained within $B$, and hence $B(u) \subseteq B$. If $p$ is a vertex of $B$, then $p,p+(0,1) \in \gamma^w \cap \gamma^v$, and hence $p, p+(0,1) \in \gamma^u$, giving $B \subseteq B(u)$. Thus $B = B(u)$, and by Corollary \ref{polytopes}, $u \sim_2 v$.

\noindent(ii) By definition $\{(i, \alpha^{w \wedge v}_i), (i, \alpha^{w \vee v}_i)\} = \{ (i, \alpha^w_i), (i, \alpha^v_i)\}$, and so one immediately has that $A(w \wedge v), A(w \vee v) \subseteq A$. Conversely, if $(i,j)$ is a vertex of $A$, then this point must belong to both $\alpha^w$ and $\alpha^v$, giving $\alpha^{w}_i = \alpha^{v}_i = j$, and hence $\alpha^{w\wedge v}_i = \alpha^{w\vee v}_i = j$, from which it follows that  $A \subseteq A(w \wedge v), A(w \vee v)$. The same reasoning shows that $B=B(w \wedge v)=B(w\vee v)$, whence the conclusion follows from Corollary \ref{polytopes}.
\end{proof}

\begin{corollary}\label{lattice_class}
Each $\sim_2$-equivalence class of $\{a,b\}^+$ is an interval of the lattice $W(\ell_a,\ell_b)$ with respect to the order $\preceq$ for some $\ell_a,\ell_b \in \mathbb{N}$. If $w$ is the unique minimal element in its equivalence class, then $\widetilde{w}$ is the unique maximal element in its class. 
\end{corollary}

\begin{corollary}\label{cor:chain}
Suppose $w,v \in W(\ell_a,\ell_b)$. Then $w \sim_2 v$ if and only if there exists a non-negative integer $k$ and a sequence of words $u(i)$, for $i=0, \ldots, k$ such that:
\begin{itemize}
\item[(i)] $w=u(0) \sim_2 u(1) \sim_2 \cdots \sim_2 u(k)=v$; and
\item[(ii)] $u(i) \leftrightarrow u(i+1)$ for all $i$.
 \end{itemize} 
In particular, $w$ is an isoterm if and only if $w \not\sim_2 w'$ for all $w' \leftrightarrow w$. 
\end{corollary}

For general alphabets, $\sim_2$-equivalence classes need not be connected through adjacent swaps as the following example illustrates.
\begin{example}
\label{ex:threeswap}
Consider the words
\begin{eqnarray*}
w=abc\; cba\; abc \; abc \; abc\; cba\; abc\\
v=abc\; cba\; abc \; cba \; abc\; cba\; abc.
\end{eqnarray*}
It is straightforward to verify that
\begin{eqnarray*}
A(w) =& \conv\{ (0,0,0), (1,2,2), (5,6,6), (6,6,6)\}& = A(v),\\
B(w) = &\conv\{ (1,0,0), (1,1,2), (5,5,6), (7,6,6)\}& = B(v),\\
C(w) =& \conv\{ (1,1,0), (1,1,1), (5,5,5), (7,7,6)\} &= C(v),\\
&B(w) \cap A(w) = B(w) \cap C(w) = \emptyset.&
\end{eqnarray*}
It follows from the above that $w \sim_2 v$, and moreover, if $w \leftrightarrow u$ with $w \sim_2 u$, then $u$ cannot have been obtained via a swap involving $b$, since this requires some point of $B(w)$ to lie in one of $A(w)$ or $C(w)$. Noting that $v$ cannot be reached from $w$ via \emph{adjacent} swaps involving $a$ and $c$ only, we see that the equivalence class of $w$ is not connected through adjacent swaps. (In fact, one finds that $w$ is not $\sim_2$ related to any of its neighbours and hence not connected to any other word of its $\sim_2$-class by adjacent swaps.)
\end{example}

Theorem \ref{thm:chain} gives a short geometric proof of Adjan's result that the identities in Example \ref{adjan} are minimal length identities for the bicyclic monoid. We begin this proof with the following lemma.
\begin{lemma}
\label{lem:iso4}
Let $v \in W(\ell_a, \ell_b)$ and suppose that $v$ is not an isoterm for $\ut{2}$. Then $\ell_a \geq 4$, and if $\ell_a = 4$, then $\ell_b > 6$.
\end{lemma}
\begin{proof}
By Corollary \ref{cor:chain}, there exists a word $u$ with $v \leftrightarrow u$  and $v \sim_2 u$. In other words, $\alpha^{v}$ and $\alpha^{u}$ differ by one in exactly one coordinate, say, $\alpha^{u}_i = \alpha^{v}_i + 1$ for some $0 \leq i \leq \ell_a-1$, where neither $(i, \alpha^{v}_i)$ nor $(i,\alpha^{v}_i+1)$ is a vertex of $A(v)$. By Lemma \ref{props}, $i \neq 0$ and $i \neq \ell_a-1$. Thus $\ell_a \geq 3$. If $\ell_a = 3$, this implies
$$ A(v) = \conv\{(0,\alpha^{v}_0), (2,\alpha^{v}_2)\} = \conv\{(0,\alpha^{v}_0), (2,\alpha^{v}_2), (1, \alpha^{v}_1), (1, \alpha^{v}_1+1)\}, $$
providing an immediate contradiction, since the four points above cannot be co-linear. Thus we must have $\ell_a \geq 4$. Suppose then that $\ell_a = 4$. Then either $i = 1$ or $i = 2$. By considering the reverse words if necessary, we can assume that $i = 1$ and hence
\begin{align}
A(v) &= \conv\{(0,\alpha^{v}_0), (1,\alpha^{v}_1), (1,\alpha^{v}_1+1), (2,\alpha^{v}_2), (3,\alpha^{v}_3)\} \nonumber \\
&= \conv\{(0,\alpha^{v}_0),(2,\alpha^{v}_2), (3,\alpha^{v}_3)\} \label{eqn:1.inside}
\end{align}
Since these points cannot be co-linear, $A(v)$ must be the triangle with the common three points of these two sets as vertices. Let $e = \alpha^{v}_2 - \alpha^{v}_0$ and $f = \alpha^{v}_3 - \alpha^{v}_2$. Note that $e,f \in \mathbb{N}_{\geq 0}$ as $\alpha^v$ is a staircase path. In fact $e, \alpha^v_0 >0$, since $\alpha^v_2 \geq \alpha^v_1+1 \geq \alpha^v_1 \geq \alpha^v_0$, whilst taking $\alpha^v_0=0$ yields a difference in the first block of $b$'s of the two words $v$ and $u$ which, by Lemma~\ref{props}, would contradict $B(v)=B(u)$. We consider two cases.
\begin{itemize}[leftmargin=*]
  \item[(I)] Suppose that the point $(2,\alpha^v_2)$ lies above the line segment $\conv\{(0,\alpha^v_0), (3,\alpha^v_3)\}$. It follows from (\ref{eqn:1.inside}) that $e$ and $f$ must satisfy 
$$\lceil \frac{e+f}{3} \rceil +1 \leq \lfloor \frac{e}{2} \rfloor,$$
which in turn implies $e \geq 6+2f$. Then $\ell_b \geq \alpha^{v}_0+e+f \geq \alpha^v_0 + 6>6$.
  \item[(II)] Suppose that the point $(2,\alpha^v_2)$ lies below the line segment $\conv\{(0,\alpha^v_0), (3,\alpha^v_3)\}$. It follows from (\ref{eqn:1.inside}) that $e$ and $f$ must satisfy 
$$\lceil \frac{e}{2}\rceil + 1 \leq \lfloor \frac{e+f}{3} \rfloor.$$
Since $e \neq 0$,  we see that $\ell_b \geq \alpha^{v}_0+e+f \geq \alpha^v_0 + 3 \lceil \frac{e}{2}\rceil + 3 \geq \alpha^v_0 +6>6$.
\end{itemize}
\end{proof}

\begin{theorem}\label{thm:adjan}
 Let $w \in W(\ell_a, \ell_b)$. If $\ell_a+\ell_b \leq 9$, then $w$ is an isoterm for the bicyclic monoid. Moreover, the only identities for the bicyclic monoid involving words of length $10$  are given by
\begin{eqnarray*}
abba \; ab \; abba &\sim_2 & abba \; ba \; abba,\;\;\; baab \; ab \; baab \sim_2  baab \; ba \; baab,\\
abba \; ab \; baab &\sim_2& abba \; ba \; baab,\;\;\; baab \; ab \; abba \sim_2 baab \; ba \; abba.
\end{eqnarray*}
\end{theorem}

\begin{proof}
Suppose $w, v \in W(\ell_a, \ell_b)$ with $w \sim_2 v$. By Lemma \ref{lem:iso4},  $\min(\ell_a,\ell_b) \geq 4$ and $\min(\ell_a,\ell_b) =4$ implies $\max(\ell_a,\ell_b) > 6$. Thus $\ell_a+\ell_b \geq 10$. Suppose $\ell_a + \ell_b = 10$. By Lemma \ref{lem:iso4}, we must have $\ell_a = \ell_b = 5$. Enumerating all possible configurations of such words (using Algorithm \ref{alg:list.identities2} from Section \ref{sec:alg}, for example) give us the desired identities. Figure \ref{fig:adjan.5} in Example \ref{adjan} visualises two of these identities.
\end{proof}

We now give an explicit construction of the minimal and maximal elements (with respect to $\preceq$ ) of a given $\sim_2$ equivalence class based on the signature of the class.  As we shall discuss in Section \ref{sec:alg}, this construction gives a significant speedup to various algorithms for two-letter identities in $\ut{n}$. Due to Corollary~\ref{lattice_class}, it is sufficient to construct the maximum element. 

First we introduce some notations. 
For each point $p \in \mathbb{N}_{\geq 0}^2$ let us write $x(p)$ and $y(p)$ to denote the first and second coordinates of $p$, respectively. For $p \trianglelefteq q$ we write $[p,q]$ to denote the set of all points $r \in \mathbb{N}_{\geq 0}^2$ such that $p \trianglelefteq r \trianglelefteq q$. 
Similarly, we write $(p,q]$, $[p,q)$ and $(p,q)$ to denote the corresponding intervals of $\mathbb{N}_0^2$ with respect to the partial order $\trianglelefteq$ which exclude one or both of the end points.  

Let $(A,B)$ be the signature of some word in $W(\ell_a,\ell_b)$, and let $\mathcal{V}$ be the monotonically increasing sequence 
$$(0,0) = p_1 \triangleleft p_2 \triangleleft \dots \triangleleft p_M \triangleleft p_{M+1} = (\ell_a,\ell_b)$$ 
consisting of the vertices of $A$ and the vertices of $B$, together with the point $(\ell_a,\ell_b)$.
Define the label function $\mathsf{label}: \mathcal{V} \to \{a,b,\emptyset\}$, where $\lab(p_i) = a$ if $p_i$ is a vertex of $A$; $\lab(p_i) = b$ if $p_i$ is a vertex of $B$, and $\lab(p_i) = \emptyset$ if $p_i = (\ell_a,\ell_b)$.
If $\gamma$ is a staircase path with signature $(A,B)$, then $\gamma$ must go through all points in $\mathcal{V}$. In particular, 
$$\gamma = \bigcup_{i=1}^M \gamma(i),$$ 
where $\gamma(i)$ is a staircase path from $p_i$ to $p_{i+1}$. We shall construct the maximal path $\overline{\gamma}$ by constructing $\overline{\gamma}(i)$ for each $i$. This is done through iterative calls to the procedure \texttt{MaxSegment}, each of which constructs a segment of $\overline{\gamma}(i)$. Figures \ref{fig:divide} and \ref{fig:conquer} illustrate the idea. For proof convenience, the function \texttt{MaxPath} returns both the path $\overline{\gamma}(i)$ and an extension of the function $\mathsf{label}$ to the points of $\overline{\gamma}(i)$.

\begin{algorithm}\caption{\texttt{Maxpath}} \label{alg:maxpath}
\flushleft 
\textbf{Input}: points $p_i,p_{i+1} \in \mathcal{V}$, $\lab(p_i),\lab(p_{i+1})$ \\
\textbf{Output}: $\overline{\gamma}(i) \subset \mathbb{N}^2_{\geq 0}$ and a function $\lab_i: \overline{\gamma}(i) \to \{a,b,\emptyset\}$ 
\begin{algorithmic}[1]
\State Set $p \gets p_i$, $\lab_i(p) \gets \lab(p_i)$, $\overline{\gamma}(i) \gets \{p\}$
\While{$p \neq p_{i+1}$}
  \State $(p', \lab_i(p')) \gets \mathsf{MaxSegment}(p,p_{i+1},\lab_i(p),\lab(p_{i+1}))$
  \State $\overline{\gamma}(i) \gets \overline{\gamma}(i) \cup [p,p']$
  \For{$q \in (p,p')$} $\lab_i(q) \gets \lab_i(p)$ \EndFor
  \State $p \gets p'$, $\lab_i(p) \gets \lab_i(p')$
\EndWhile 
\State \textbf{return } $(\overline{\gamma}(i),\lab_i)$
\Procedure{MaxSegment}{$p,p_{i+1},\lab(p),\lab(p_{i+1})$}
\If{$y(p) = y(p_{i+1})$ or $x(p) = x(p_{i+1})$} 
  \State $p' \gets p_{i+1}$, $\lab(p') \gets \lab(p_{i+1})$
\EndIf
\If{$\lab(p) = b$} 
  \State $Y \gets \{k: y(p) < k \leq y(p_{i+1}), (x(p),k) \in A, (x(p),k-1) \in B\}$ \label{ln:a-set}
  \State $y \gets \max Y$ 
  \State $p' \gets (x(p), y)$, $\lab(p') \gets a$
\EndIf
\If{$\lab(p) = a$} 
  \State $X \gets \{k: x(p) < k \leq x(p_{i+1}), (k,y(p)) \in B, (k-1,y(p)) \in A, (k,y(p)+r) \in A \mbox{ for some } r > 0\}$ \label{ln:b-set}
  \State $x \gets \min X$.
  \State $p' \gets (x, y(p))$, $\lab(p') \gets b$
\EndIf
\State \textbf{return } $(p',\lab(p'))$
\EndProcedure
\end{algorithmic}
\end{algorithm}
\begin{definition}\label{defn:min.max}
For $i = 1, \dots, M$, let $(\overline{\gamma}(i),\lab_i)$ denote the output of the \texttt{Maxpath} algorithm with input $(p_i, p_{i+1})$. Let $\overline{\gamma} := \bigcup_{i=1}^M \overline{\gamma}(i)$, and define $\lab: \overline{\gamma} \to \{a,b,\emptyset\}$ by $\lab(p) = \lab_i(p) \mbox{ if } p \in \overline{\gamma}(i).$ 
\end{definition}

\begin{theorem}\label{thm:min.max}
Let $w \in W(\ell_a, \ell_b)$ be a word with signature $(A, B)$. The path of the unique maximal word in the $\sim_2$-class of $w$ is given by $\overline{\gamma}$. 
\end{theorem}

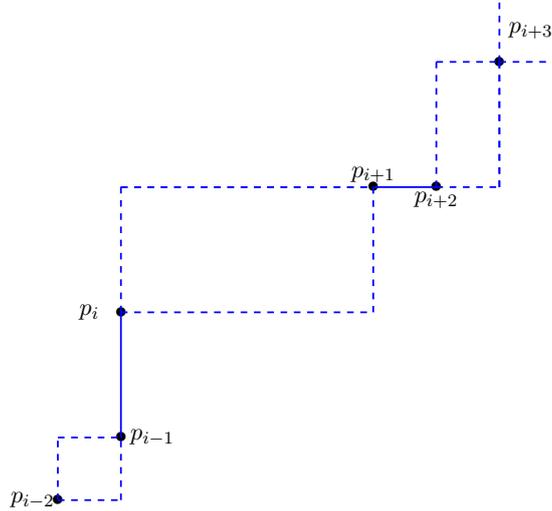
\begin{figure}[t!]
\captionsetup{width=1\linewidth}
\centering  
\resizebox{0.6\textwidth}{!}{
\begin{tikzpicture}
  \foreach \Point in {(-1,-1), (0,0), (0,2), (4,4), (5,4), (6,6)}{
    \node at \Point {\textbullet};
}
\draw[blue, thick, dashed](-1,-1)-- (-1,0);
\draw[blue, thick, dashed](-1,-1)-- (0,-1);
\draw[blue, thick, dashed](-1,0)-- (0,0);
\draw[blue, thick, dashed](0,-1)-- (0,0);
\draw[blue, thick](0,0)-- (0,2);
\draw[blue, thick](4,4)-- (5,4);
\draw[blue, thick, dashed](0,2)-- (0,4);
\draw[blue, thick, dashed](0,4)-- (4,4);
\draw[blue, thick, dashed](0,2)-- (4,2);
\draw[blue, thick, dashed](4,4)-- (4,2);
\draw[blue, thick, dashed](0,4)-- (4,4);
\draw[blue, thick, dashed](5,4)-- (6,4);
\draw[blue, thick, dashed](5,4)-- (5,6);
\draw[blue, thick, dashed](5,6)-- (6,6);
\draw[blue, thick, dashed](6,4)-- (6,6);
\draw[blue, thick, dashed](6,4)-- (6,6);
\draw[blue, thick, dashed](6,6)-- (6,7);
\draw[blue, thick, dashed](6,6)-- (7,6);

\node at (-1.4, -1) {$p_{i-2}$};
\node at (0.5, 0) {$p_{i-1}$};
\node at (-0.5, 2) {$p_i$};
\node at (4, 4.2) {$p_{i+1}$};
\node at (5, 3.8) {$p_{i+2}$};
\node at (6.5, 6.5) {$p_{i+3}$};
\end{tikzpicture}
}
\caption{Algorithm \textsf{Maxpath} constructs a staircase path from $(0,0)$ to $(\ell_a, \ell_b)$, as a union of staircase paths between the consecutive vertex points of polytopes $A$ and $B$.}
\label{fig:divide}
\end{figure}
\begin{figure}[b
]  
\captionsetup{width=1\linewidth}   
\begin{subfigure}[t]{0.45\linewidth} 
\centering 
\resizebox{\linewidth}{!}{
\begin{tikzpicture}
      \foreach \a in {0,0.5,...,4}{
     \draw[help lines, line width=0.05pt,]
     (\a,0)-- (\a,2);}

  \foreach \a in {0,0.5,...,2}{
     \draw[help lines, line width=0.05pt,]
 (0,\a)-- (4,\a);
   }
  
\foreach \Point in { (1.5, 1.5)}{
    \node at \Point {$\bullet$};
}

\foreach \Point in { (1.5,0.5)}{
    \node at \Point {$\circ$};
}
\foreach \Point in { (1.5,1)}{
    \node at \Point {$\circ$};
}

\draw[blue, thick, dashed](0,0)-- (0,2);
\draw[blue, thick, dashed](0,2)-- (4,2);
\draw[blue, thick, dashed](0,0)-- (4,0);
\draw[blue, thick, dashed](4,2)-- (4,0);

\node at (0, -0.2) {$p_i$};
\node at (4, 2.2) {$p_{i+1}$};
\node at (1.7, 0.5) {$p$};
\node at (1.7, 1.5) {$p'$};

\draw[black, thick](0,0)-- (1,0);
\draw[black, thick](1,0)-- (1,0.5);
\draw[black, thick](1,0.5)-- (1.5,0.5);
\draw[blue, thick](1.5,0.5)-- (1.5,1.5);
\end{tikzpicture} 
} 
\caption{At $p$ with $\lab(p)=b$, move North as much as possible to point $p'$ such that $p' \in A$ and $p'-(0,1) \in B$. Label $p'$ as $a$, and all other points between $p$ and $p'$ as $b$.}
\label{fig:subfig_2}
\end{subfigure}
\hfill %
\begin{subfigure}[t]{0.45\linewidth}
\centering
\resizebox{\linewidth}{!}{
\begin{tikzpicture}
      \foreach \a in {0,0.5,...,4}{
     \draw[help lines, line width=0.05pt,]
     (\a,0)-- (\a,2);}

  \foreach \a in {0,0.5,...,2}{
     \draw[help lines, line width=0.05pt,]
 (0,\a)-- (4,\a);
   }
  
\foreach \Point in { (1.5, 1.5)}{
    \node at \Point {$\bullet$};
}
\foreach \Point in { (2, 1.5)}{
    \node at \Point {$\bullet$};
}

\foreach \Point in { (2.5,1.5)}{
    \node at \Point {$\circ$};
}
\foreach \Point in { (2.5,2)}{
    \node at \Point {$\bullet$};
}

\draw[blue, thick, dashed](0,0)-- (0,2);
\draw[blue, thick, dashed](0,2)-- (4,2);
\draw[blue, thick, dashed](0,0)-- (4,0);
\draw[blue, thick, dashed](4,2)-- (4,0);

\node at (0, -0.2) {$p_i$};
\node at (4, 2.2) {$p_{i+1}$};
\node at (1.25, 1.25) {$p$};
\node at (2.5, 1.25) {$p'$};

\draw[black, thick](0,0)-- (1,0);
\draw[black, thick](1,0)-- (1,0.5);
\draw[black, thick](1,0.5)-- (1.5,0.5);
\draw[black, thick](1.5,0.5)-- (1.5,1.5);
\draw[blue, thick](1.5,1.5)-- (2.5,1.5);
\end{tikzpicture}
}
\caption{At $p$ with $\lab(p) =a$, move East as little possible to the first point $p'$ such that $p' \in B$, $p' - (1,0) \in A$, and $p' + (0,r) \in A$ for some $r > 0$. Label $p'$ as $b$ and all other points between $p$ and $p'$ as $a$.
}
\label{fig:subfig_3}
\end{subfigure} 
\hfill %
\begin{subfigure}[t]{0.45\linewidth}
\centering
\resizebox{\linewidth}{!}{
\begin{tikzpicture}
      \foreach \a in {0,0.5,...,4}{
     \draw[help lines, line width=0.05pt,]
     (\a,0)-- (\a,2);}
  
  \foreach \a in {0,0.5,...,2}{
     \draw[help lines, line width=0.05pt,]
 (0,\a)-- (4,\a);
   }
  
\foreach \Point in { (2.5, 2)}{
    \node at \Point {$\bullet$};
}

\foreach \Point in { (2.5,1.5)}{
    \node at \Point {$\circ$};
}

\foreach \Point in { (4,2)}{
    \node at \Point {$\bullet$};
}

\draw[blue, thick, dashed](0,0)-- (0,2);
\draw[blue, thick, dashed](0,2)-- (4,2);
\draw[blue, thick, dashed](0,0)-- (4,0);
\draw[blue, thick, dashed](4,2)-- (4,0);

\node at (0, -0.2) {$p_i$};
\node at (4, 2.2) {$p_{i+1}$};
\node at (2.7, 2.2) {$p$};

\draw[black, thick](0,0)-- (1,0);
\draw[black, thick](1,0)-- (1,0.5);
\draw[black, thick](1,0.5)-- (1.5,0.5);
\draw[black, thick](1.5,0.5)-- (1.5,1.5);
\draw[black, thick](1.5,1.5)-- (2,1.5);
\draw[black, thick](2.5,1.5)-- (2,1.5);
\draw[black, thick](2.5,1.5)-- (2.5,2);
\draw[blue, thick](2.5,2)-- (4,2);
\end{tikzpicture}  
}
\caption{At the boundary follow the unique line segment to point $p_{i+1}$. Give all points between $p$ and $p_{i+1}$ the same label as that of $p$.}
\label{fig:subfig_4}
\end{subfigure} 
\hfill
\begin{subfigure}[t]{0.45\linewidth}
\centering
\resizebox{\linewidth}{!}{
\begin{tikzpicture}
      \foreach \a in {0,0.5,...,4}{
     \draw[help lines, line width=0.05pt,]
     (\a,0)-- (\a,2);}
  
  \foreach \a in {0,0.5,...,2}{
     \draw[help lines, line width=0.05pt,]
 (0,\a)-- (4,\a);
   }

\draw[blue, thick, dashed](0,0)-- (0,2);
\draw[blue, thick, dashed](0,2)-- (4,2);
\draw[blue, thick, dashed](0,0)-- (4,0);
\draw[blue, thick, dashed](4,2)-- (4,0);

\node at (0, -0.2) {$p_i$};
\node at (4, 2.2) {$p_{i+1}$};

%top path
\draw[black, thick](0,0)-- (1,0);
\draw[black, thick](1,0)-- (1,0.5);
\draw[black, thick](1,0.5)-- (1.5,0.5);
\draw[black, thick](1.5,0.5)-- (1.5,1.5);
\draw[black, thick](1.5,1.5)-- (2,1.5);
\draw[black, thick](2.5,1.5)-- (2,1.5);
\draw[black, thick](2.5,1.5)-- (2.5,2);
\draw[black, thick](2.5,2)-- (4,2);

%bottom path
\draw[red!60, thick](0,0)-- (1.5,0);
\draw[red!60, thick](1.5,0)-- (1.5,1);
\draw[red!60, thick](1.5,1)-- (3,1);
\draw[red!60, thick](3,1)-- (3,1.5);
\draw[red!60, thick](3,1.5)-- (4,1.5);
\draw[red!60, thick](4,1.5)-- (4,2);

\end{tikzpicture}  
}
\caption{In the proof of Theorem \ref{thm:min.max}, the function $\lab_i$ constructed together with points on the original path $\gamma$ (in red) are used to show that each call to \textsf{MaxSegment} has well-defined output.}
\label{fig:subfig_5}
\end{subfigure} 
\caption{Procedure \textsf{MaxSegment} in pictures.}
\label{fig:conquer}
\end{figure} 

\begin{proof}
Let $\gamma$ be any path with signature $(A,B)$, and let $\gamma(i)$ be its subpath from $p_i$ to $p_{i+1}$. Let $\alpha$ and $\beta$ denote the respective $a$- and $b$-heights of $\gamma$.  For each $i = 1, \dots, M-1$, we shall show by induction on the number of iterations of \textsf{MaxSegment} that after each iteration with input $p, p_{i+1}$:
\begin{enumerate}
  \item[(i)] Point $p'$ is well-defined and $[p,p']$ is either a North or East segment.
  \item[(ii)] For each point $q \in (p,p'] \setminus \mathcal{V}$
\begin{align}
\lab(q)=a \;\;\;&\Rightarrow \;\;\;\alpha_{x(q)} \leq y(q) \mbox{ and } q \in A \label{eqn:alpha.q}, \quad \mbox{and}  \\
\lab(q)=b  \;\;\;&\Rightarrow \;\;\;\beta_{y(q)} \geq x(q) \mbox{ and } q \in B. \label{eqn:beta.q}
\end{align}
\end{enumerate}
Once this statement has been established, the remaining steps of the proof proceed as follows. Condition (i) implies that \textsf{Maxpath} terminates in finitely many step; thus $\overline{\gamma}(i)$ is well-defined. By construction, $\overline{\gamma}(i)$ is a union of North and East segments, and so a staircase path from $p_i$ to $p_{i+1}$. Thus, $\overline{\gamma} = \bigcup_{i=1}^m \overline{\gamma}(i)$ is a staircase path from $(0,0)$ to $(\ell_a,\ell_b)$. Note that \texttt{Maxpath} does not change the labels of points in $\mathcal{V}$, and thus the \textsf{label} function output by \texttt{Maxpath} is well-defined and extends the definition on the vertices $\mathcal{V}$ to the entire of $\overline{\gamma}$. Let 
$$\overline{\alpha} = \{q \in \overline{\gamma}: \lab(q) = a\}, \quad \overline{\beta} = \{q \in \overline{\gamma}: \lab(q) = b\}.$$
Since $\overline{\gamma}$ is a staircase path, by construction of $\lab$, $\overline{\alpha}$ is the $a$-height of $\overline{\gamma}$, and $\overline{\beta}$ is the $b$-height of $\overline{\gamma}$. Therefore, condition (ii) implies that $\gamma \preceq \overline{\gamma}$. Furthermore, (\ref{eqn:alpha.q}) implies $\overline{\alpha} \subseteq A$, so $A(\overline{\gamma}) \subseteq A$. By definition of $\lab$ on points in $\mathcal{V}$, all vertices of $A$ are in $\overline{\alpha}$, so $A(\overline{\gamma}) \supseteq A$. Thus $A = A(\overline{\gamma})$. A similar argument applied to $B$ and $\overline{\beta}$ implies $B = B(\overline{\gamma})$. So $\overline{\gamma}$ has signature $(A,B)$, and it dominates any path $\gamma$ with signature $(A,B)$. Thus it is the desired maximal path. 

Now we prove (i) and (ii) by induction on the number of iterations of \textsf{MaxSegment}. Consider the first iteration, where $(p_i,p_{i+1},\lab(p_i),\lab(p_{i+1}))$ is the input to \textsf{MaxSegment}. Let $p = p_i$. There are three cases:
\begin{itemize}[leftmargin=*]
  \item[(I)] If $p$ shares a coordinate with $p_{i+1}$:  Then \textsf{MaxSegment} has well-defined output $p_{i+1}$, and (i) clearly holds. If $\lab(p) = a$, then $p$ is a vertex of $A$, and so $[p,p_{i+1}]$ must be an East segment. Thus $y(p) = y(p_{i+1})$ and it follows from the fact that $\gamma$ is a staircase path that for any $q \in (p,p_{i+1})$, we have $\alpha(i)_{x(q)} = y(q) = y(p)$, giving (\ref{eqn:alpha.q}). If $\lab(p) = b$, then by the same reasoning, $[p,p_{i+1}]$ is a North segment, and $\beta(i)_{y(q)} = x(q) = x(p)$ for all $q \in (p,p_{i+1})$, giving (\ref{eqn:beta.q}).
\item[(II)] If $p$ does not share a coordinate with $p_{i+1}$ and $\lab(p) = b$: The set $Y$ defined in line \ref{ln:a-set} is nonempty, since $\alpha_{x(p)} \in Y$. Thus the output $p'$ is well-defined,  and $[p,p']$ is a North segment. This proves (i). For (ii), note that $p' \in A$ is the unique point in this segment with label $a$ and since we chose $p'$ to have maximal $y$ coordinate amongst points in $Y$, we have $\alpha_{x(p')} = \alpha_{x(p)} \leq y(p')$ , giving (\ref{eqn:alpha.q}). By construction, $p \trianglelefteq  p'-(0,1) \in B$, so $q \in B$ for each $q \in (p,p')$ by convexity. Since $\gamma$ is a staircase path, $\beta_{y(q)} \geq \beta_{y(p)} \geq x(p) = x(q)$ for all points  $q \in (p,p')$, proving (\ref{eqn:beta.q}).
  \item[(III)] If $p$ does not share a coordinate with $p_{i+1}$ and $\lab(p) = a$: The set $X$ defined in line \ref{ln:b-set} is nonempty, since $\beta_{y(p)} \in X$. Thus the output $p'$ is well-defined, and $[p,p']$ is an East segment. This proves (i). For (ii), note that $p' \in B$ is the unique point in this segment with label $b$ and since we chose $p'$ to have minimal $x$ coordinate amongst points in $X$, we have $\beta_{y(p')} = \beta_{y(p)} \geq x(p')$, giving (\ref{eqn:beta.q}). By construction, $p \trianglelefteq  p'-(0,1) \in A$, so $q \in A$ for each $q \in (p,p')$ by convexity. If \eqref{eqn:alpha.q} did not hold for some $q \in (p,p')$, then we would have $\alpha_{x(q)} >y(q) =y(p)$. But this contradicts that $\gamma$ is a Northeast staircase path passing through the points $(x(q), \alpha_{x(q)})$ and $(\beta_{y(p)}, y(p))$, since by assumption the former lies directly North of $q$, whilst the latter lies directly to the East of $q$, since $x(q) < x(p') \leq \beta_{y(p)}$. 
\end{itemize}

Now suppose the induction hypothesis holds for the first $j$ iterations of \textsf{MaxSegment}.
Consider the $(j+1)$-th iteration. Let $(p,\lab(p))$ be the output of \textsf{MaxSegment} in the $j$-th iteration. If $p = p_{i+1}$, then we are done. Otherwise $p$ is \emph{not a vertex}. Run \textsf{MaxSegment} with input $(p,p_{i+1},\lab(p),\lab(p_{i+1}))$. Again, we have three cases.
\begin{itemize}[leftmargin=*]
\item[(I)] If  $p$ shares a coordinate with $p_{i+1}$: Then \textsf{MaxSegment} outputs $p_{i+1}$, and  (i) holds.  Let $q$ be a point in $(p,p_{i+1})$.
  
If $\lab(p) = a$, then $[p,p_{i+1}]$ is an East segment and $\lab(q) =a$. Thus $y(q) = y(p_{i+1}) \geq \alpha_{x(q)}$. Since $p_{i+1}$ is a vertex point lying on the path $\gamma$, one easily verifies that
$$ R:= \{p_{i+1}, p_{i+1}-(1,0), p_{i+1}+(0,1)\} \cap A \neq \emptyset.$$ 
Noting that $q$ lies in the convex hull of $p,r$ and$(x(q), \alpha_{x(q)})$ for each $r \in R$, we obtain $q \in A$. This proves (\ref{eqn:alpha.q}) for case (I).

If $\lab(p) = b$, then $[p,p_{i+1}]$ is a North segment, $\lab(q)=b$, and
$x(q)= x(p_{i+1}).$ As $\gamma(i)$ is a staircase path from $p_i$ to $p_{i+1}$, 
$\beta_{y(p)} \leq \beta_{y(q)} \leq x(p_{i+1})$.
Since the induction hypothesis holds at the non-vertex point $p$, we have $\beta_{y(p)} \geq x(p) = x(p_{i+1})$, giving  $x(q) = x(p_{i+1}) = \beta_{y(q)}$. Since $p_{i+1}$ is a vertex point lying on the path $\gamma$, one easily verifies
$$R:= \{p_{i+1}, p_{i+1}+(1,0), p_{i+1}-(0,1)\} \cap  B \neq \emptyset.$$
Noting that $q$ lies in the convex hull of $p,r$ and $(\beta_{y(q)}, y(q))$
%$(x(q), \alpha_{x(q)})$
for each $r \in R$, we obtain $q \in B$. This proves (\ref{eqn:beta.q}) for case (I). 
\item[(II)] If $p$ does not share a coordinate with $p_{i+1}$ and $\lab(p) = b$: Then $p$ is the last point of an East move by \textsf{MaxSegment}. Thus $p$ arises by taking the minimum of a set of the form $X$ from line \ref{ln:b-set} (for the previous value of $p$). This implies that $y(p)+1$ lies in the set $Y$ from line \ref{ln:a-set} (for the current value of $p$). Thus $p'$ is well defined and (i) holds.  Using the inductive hypothesis we have $\beta_{y(p)} \geq x(p)$ and hence $\alpha_{x(p)} \leq y(p)$. The reasoning of the base step, case (II) then applies to show that (ii) holds.
  \item[(III)] If $p$ does not share a coordinate with $p_{i+1}$ and $\lab(p) = a$: Then $p$ is the last point of a North move by \textsf{MaxSegment}. We now argue that $\beta_{y(p)} \in X$. By the inductive hypothesis, $\alpha_{x(p)} \leq y(p)$. Therefore, $\beta_{y(p)} \geq x(p)$. This last inequality must be strict, or else $p$ belongs to $\beta$, which would contradict the maximality of the $y$ coordinate of $p$ obtained in the previous step. Let $q' = (\beta_{y(p)},y(p))$ be the point on $\gamma$ at height $y(p)$. It follows from the definition of $\beta$ that $q' \in B$, $\beta_{y(p)} \leq x(p_{i+1})$, and that the path $\gamma$ continues to the North immediately after $q'$. Since $q' \triangleleft p_{i+1}$ and $p_{i+1}$ is a vertex of either $A$ or $B$, there must be at least one more East step in $\gamma$ after its North segment from $q'$. Thus $q'+(0,r) \in A$ for some $r > 0$. It remains to show that $q'-(1,0) \in A$. Let $q^- \in \alpha$ be the point with $x$-coordinate $x(q')-1$, and $q^+ \in \alpha$ be the point with $x$-coordinate $x(q')$. That is, $q^-$ is the point of $\gamma$ at which the last East move before $q'$ occurred and let $q^+ \in \alpha$ is the point of $\gamma$ at which the first East move after $q'$ occurs. It is immediate that $q'-(1,0)$ lies in the convex hull of $q^-,q^+$ and $p$, which are all points in $A$, giving $q'-(1,0) \in A$.  The reasoning of the base step, case (III) then applies to show that (ii) holds.
\end{itemize} 
This concludes the induction step, and thus establishes (i) and (ii), as desired.
\end{proof}
Dually, one may define the path of the minimal word via Corollary~\ref{lattice_class} and \texttt{Maxpath}. We denote this minimal path by $\underline{\gamma}$.
\begin{corollary}
The cardinality of the $\sim_2$ equivalence class containing $w \in W(\ell_a, \ell_b)$ is equal to the number of lattice paths bounded by $\underline{\gamma}$ and $\overline{\gamma}$.
\end{corollary}

\begin{example}\label{ex:min.max}
Consider the word 
$$w = baabbaabbabaabaaababaaba.$$ 
Figure \ref{fig:minmax}(A) shows $A(w)$ and its vertices in blue, and $B(w)$ and its vertices in orange. The boxes between the points $p_i$ and $p_{i+1}$ are shown in grey. Figure \ref{fig:minmax}(B) and \ref{fig:minmax}(C) show the paths $\underline{\gamma}$ and $\overline{\gamma}$, respectively, in black. Figure \ref{fig:minmax}(D) shows $\underline{\gamma}$ in black and $\overline{\gamma}$ in red. The minimal and maximal words are
$$ \underline{w} = baababababbaaabaababaaba, \quad \overline{w} = baabbabababaabaabaabaaba,$$ and one can easily see that the equivalence class of $w$ has size $2^5 = 32$. 
\end{example}
\begin{figure}[h!]
\begin{subfigure}[t]{0.45\linewidth} 
\centering 
\resizebox{\linewidth}{!}{\includegraphics{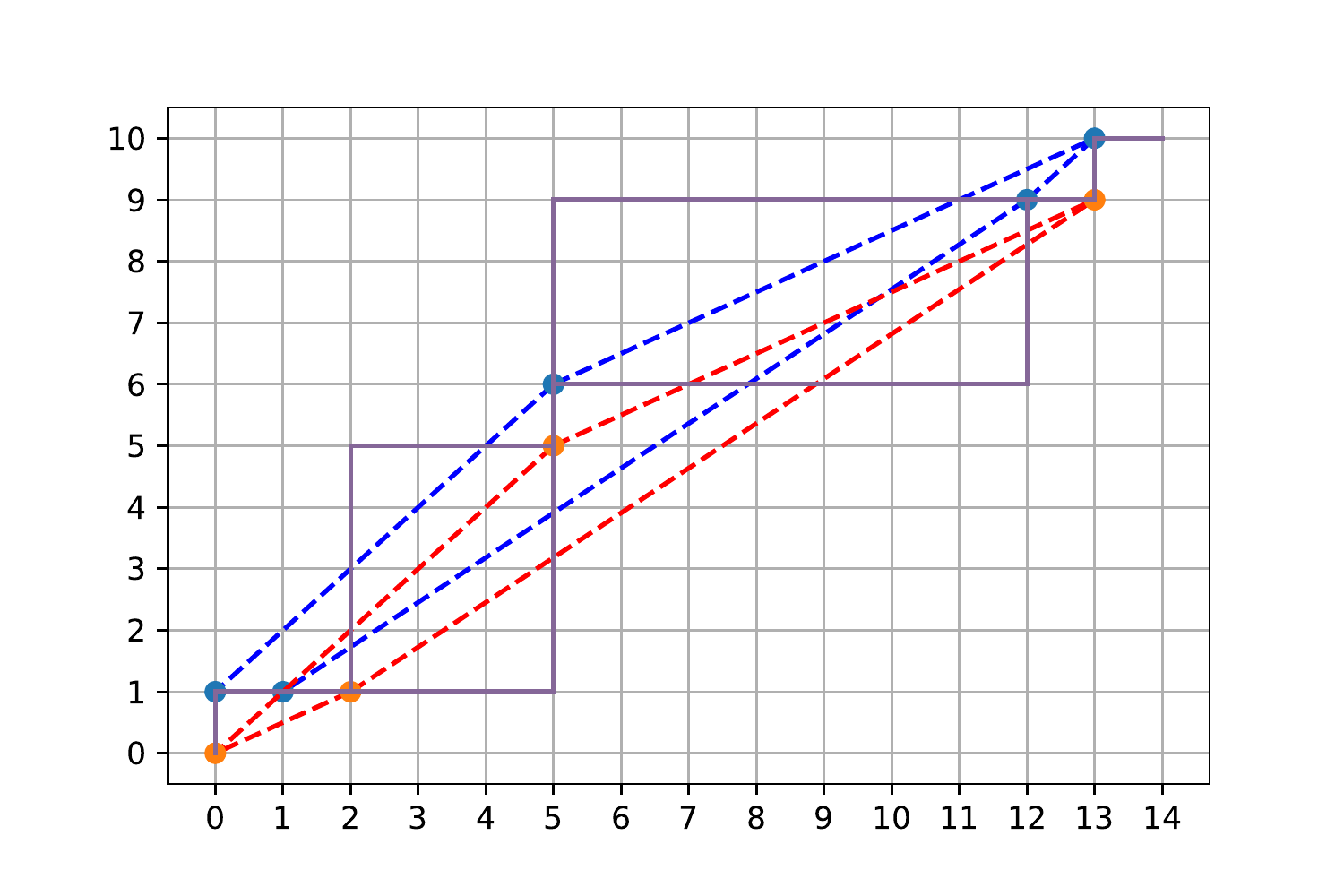}}
\caption{}
\end{subfigure}
\begin{subfigure}[t]{0.45\linewidth} 
\centering 
\resizebox{\linewidth}{!}{\includegraphics{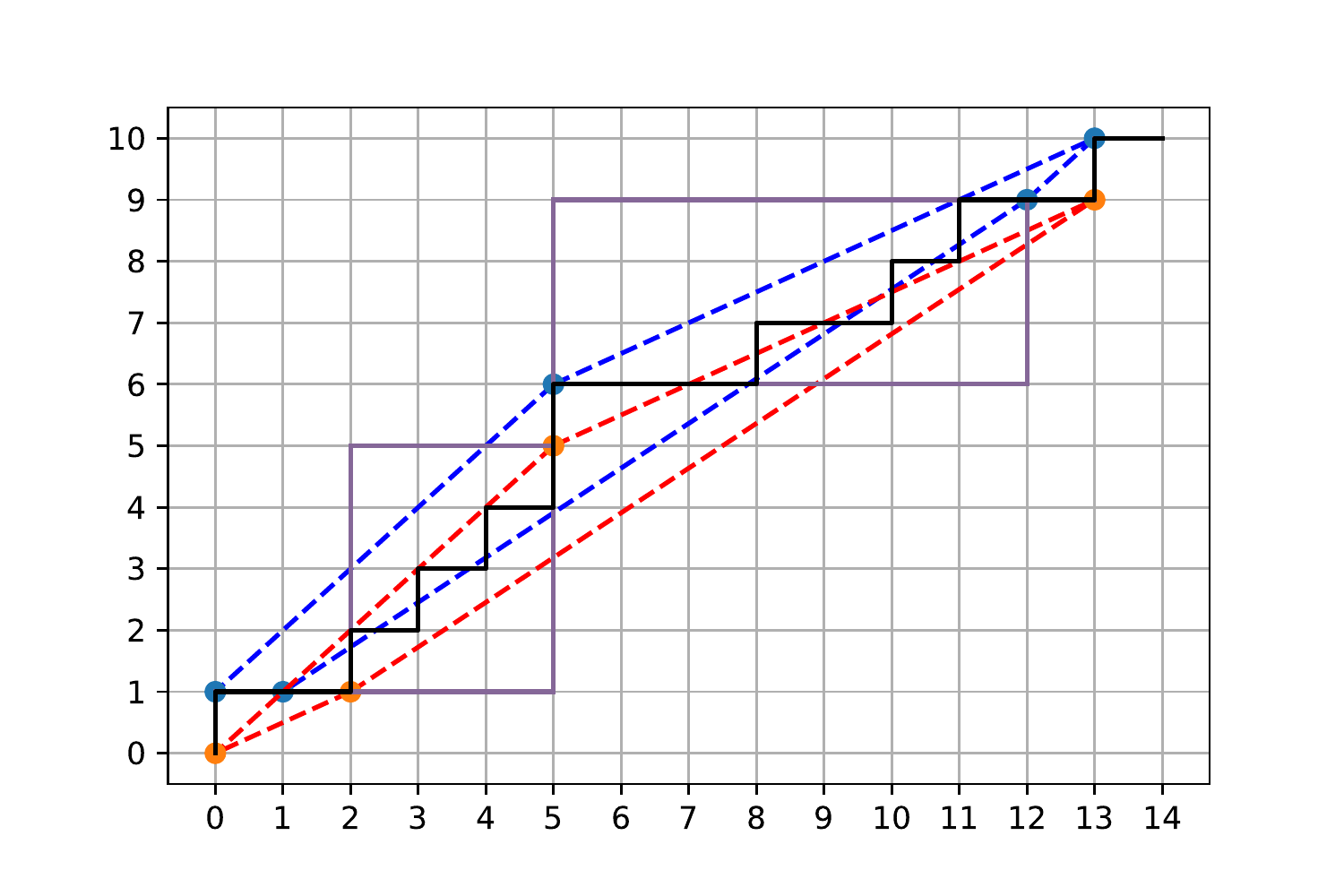}}
\caption{}
\end{subfigure}
\begin{subfigure}[t]{0.45\linewidth} 
\centering 
\resizebox{\linewidth}{!}{\includegraphics{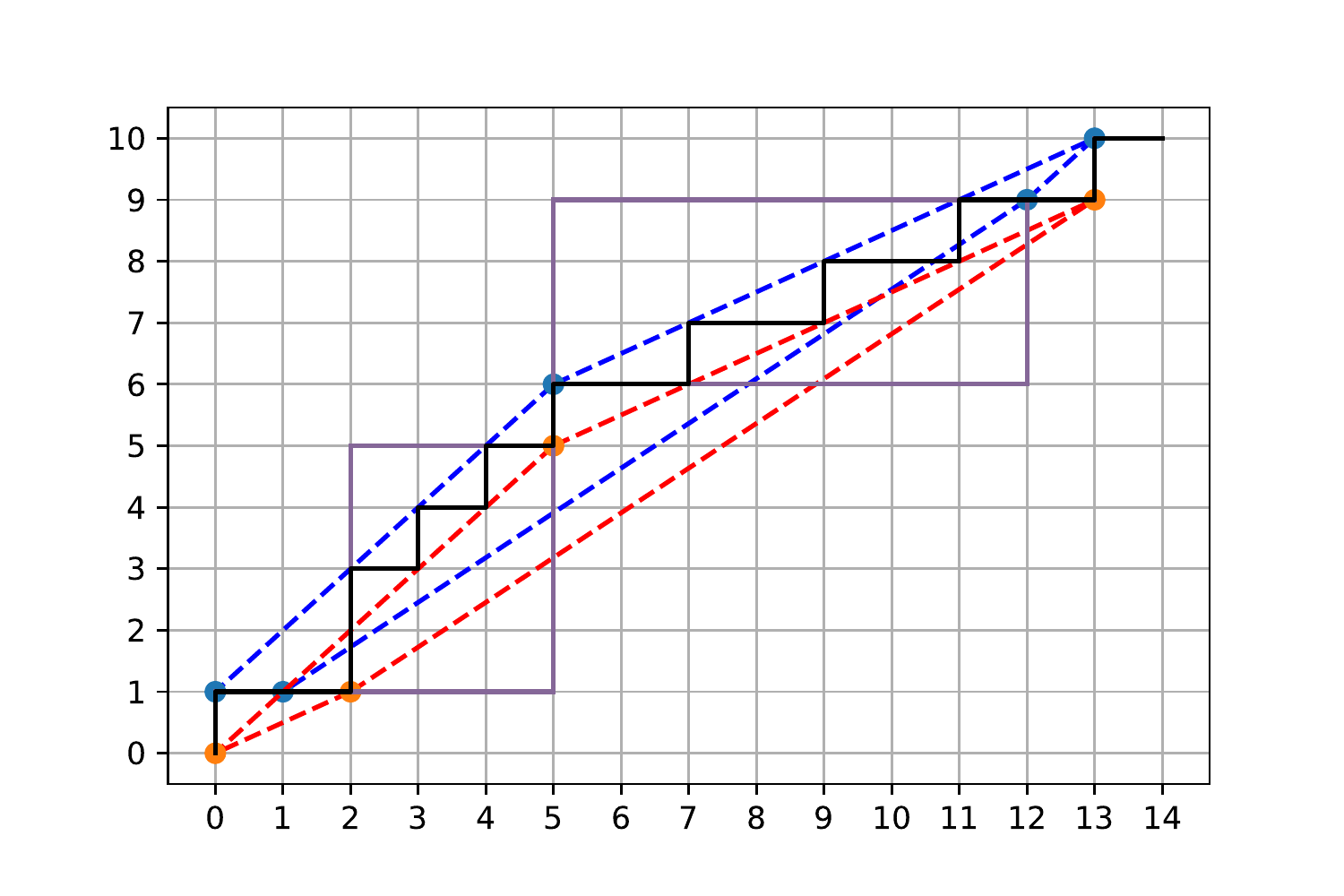}}
\caption{}
\end{subfigure}
\begin{subfigure}[t]{0.45\linewidth} 
\centering 
\resizebox{\linewidth}{!}{\includegraphics{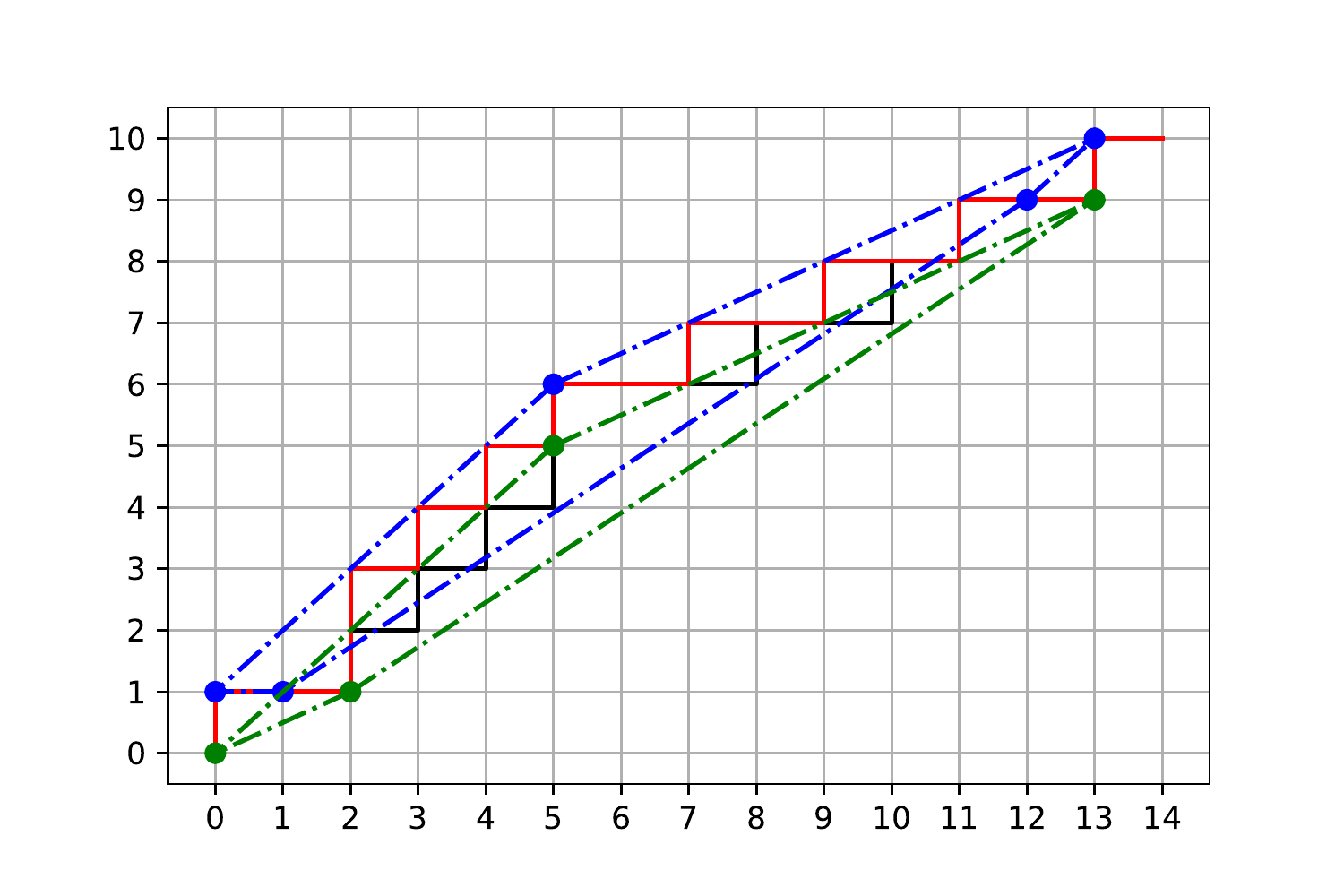}}
\caption{}
\end{subfigure}
\caption{Computing the minimal and maximal elements of the word  in Example \ref{ex:min.max} }
%(A) Signature $(A,B)$ and their vertices. (B) to (D): the paths $\underline{\gamma}$ and $\overline{\gamma}$, shown individually and then collectively.}
\label{fig:minmax}
\end{figure}

\begin{example}[The Generalized Catalan family]\label{ex:catalan}
For integers $r,k \geq 2$, define the word $w(r,k) \in W(r+k+1,r+k+1)$ by
$$ w(r,k) = (ab^k)\,(ab)^r\,(a^kb). $$
We call the $\sim_2$-equivalence class of $w(r,k)$ the generalised Catalan family $C(r,k)$. Notice that in this case, $A:=A(w(r,k))$ is a trapezium with vertices 
$$\mathcal{V}(A) = \{(0,0),(1,k),(r+1,r+k),(r+k,r+k)\},$$
while $B := B(w(r,k))$ is another trapezium with vertices
$$\mathcal{V}(B) = \{(1,0),(1,k-1),(r+1,r+k-1),(r+k+1,r+k)\}.$$
The pair of lines constructed by extending the parallel edges of $B$ are the shift by $(1,0)$ of the corresponding lines constructed from the parallel edges of $A$. The path of each word in $C(r,k)$ is thus constrained to pass through each point of $\mathcal{V}(A) \cup \mathcal{V}(B) \cup \{(2, k), (r+k+1, r+k+1)\}$, whilst remaining within the union of the two trapeze. By Theorem \ref{thm:min.max}, $w(r,k)$ is the maximal element of this equivalence class. The minimal word is given by $(ab^k)\,a^rb^r\,(a^kb)$ when $r<k$ and $(ab^k)\,a^k(ba)^{r-k}b^k\,(a^kb)$ when $r \geq k$. Figure~\ref{fig:w54} illustrates the paths of the two words and the corresponding signature for the family $C(5,4)$.

\begin{figure}[h]\captionsetup{width=1\linewidth}
\includegraphics[width=0.7\textwidth]{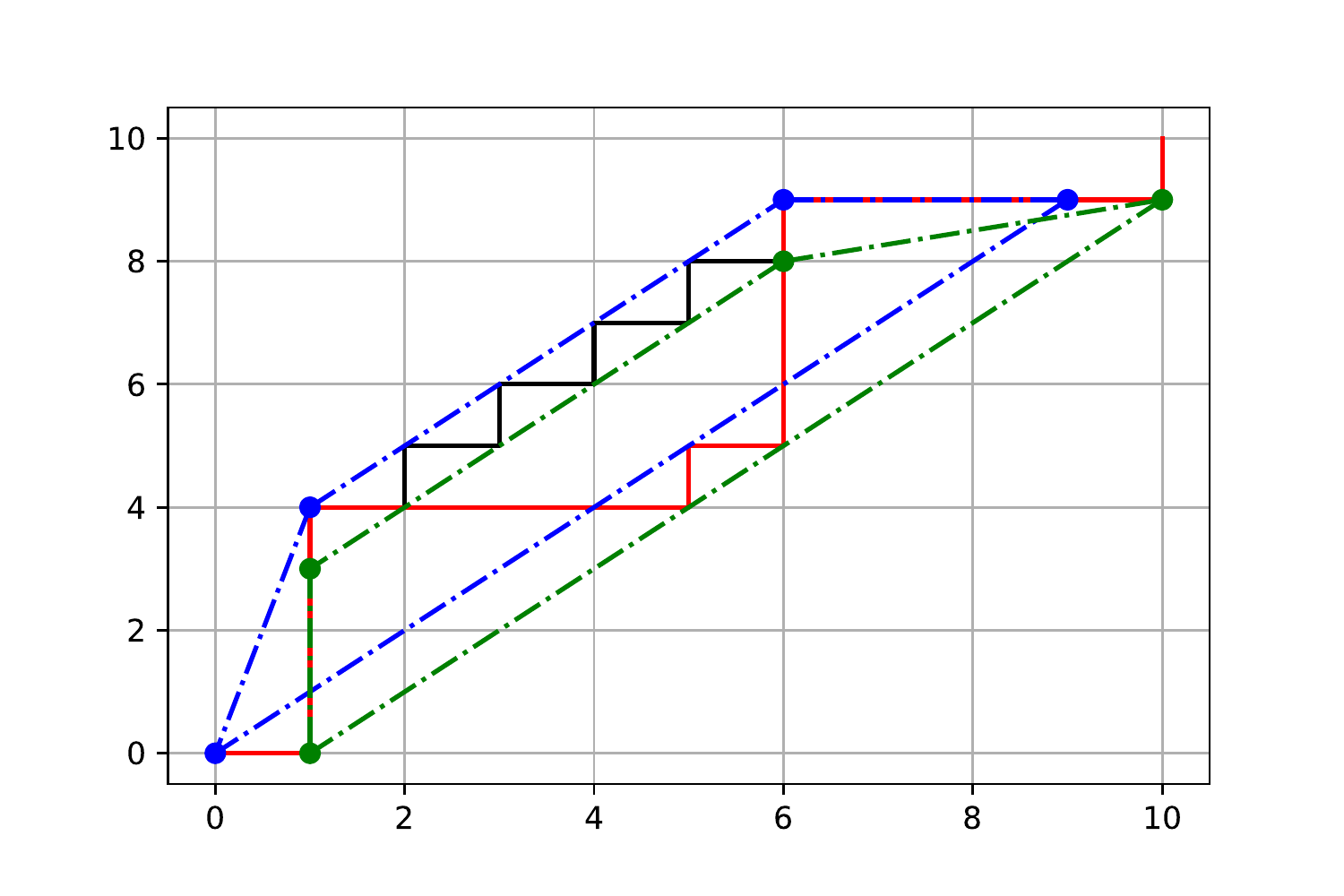}
\caption{The degree one signature of $w(5,4)$ from Example \ref{ex:catalan}, with maximal and minimal words shown. The size of the equivalence class is the number of Dyck paths of length $8$ with height at most $4$.}\label{fig:w54}
\end{figure}

The size of the equivalence class $C(r,k)$ is equal to the number of staircase paths from $(0,0)$ to $(r,r)$ weakly bounded by the lines $y=x$ and $y=x-k$, or in other words, the number of Dyck paths of length $2r$ with height at most $k$. A formula for the size of this equivalence class is therefore provided by \cite[Theorem 10.3.4]{Krattenthaler}:
$$|C(r,k)| = \frac{2^{2r+2}}{k+2} \sum_{j=1}^{\lfloor\frac{k+1}{2}\rfloor} \cos^{2r} \left(\frac{j\pi}{k+2}\right) \sin^{2}\left(\frac{j\pi}{k+2}\right).$$
For fixed $k$ and $r \to \infty$, the first order term in $|C(r,k)|$ is
$$ |C(r,k)| \sim O\left(4^r\cos^{2r}\left(\frac{\pi}{k+2}\right)\right) \sim O\left(2^{\ell}\cos^{{\ell}}\left(\frac{\pi}{k+2}\right)\right), $$
where $\ell = 2(r+k+1)$ is the word length. That is, this equivalence class grows exponentially fast in the word length. Here are some asymptotics for low values of $k$.

\begin{table}[h]
\begin{tabular}{cccccccccc}
$k$ & 2 & 3 & 4 & 5 & 6 & 7 & 8 & 9 & 10 \\ \hline
$|C(r,k)| \sim $& $2^r$ & $2.61^r$ & $3^r$ & $3.24^r$ & $3.41^r$ & $3.53^r$ & $3.61^r$ & $3.68^r$ & $3.73^r$
\end{tabular}
\end{table}
\end{example}

\section{Algorithms}
\label{sec:alg}

\subsection{Algorithms for the general case}
The \texttt{Signature} algorithm computes the degree-$d$ signature of a word. It underpins the first step in all algorithms on a general alphabet. 
To compute the degree $d$ signature of a word from definition, we need to compute the support of the polynomials $g_u^w$ defined in (\ref{eqn:gu.w}) for all $u \in \Sigma^d$. This means for each $u \in \Sigma^d$, one needs to list the ways that $u$ can appear as a scattered subword of $w$, and in each case, for each letter $s \in \Sigma$, count the number of occurrences of $s$ within each intermediary factor. To improve efficiency, instead of computing $g_u^w$ for one $u$ at a time, the \texttt{Signature} algorithm computes the monomial $\odot_{s\in\Sigma} \odot_{k=1}^{d} x(s,k)^{N_s^w(\pi_{k-1},\pi_k)}$ for each $\pi$ in (\ref{eqn:index.i}), then uses a lookup table indexed by elements of $\Sigma^d$ to associate this monomial with the correct polynomial $g_u^w$.

\begin{lemma}\label{lem:signature.complexity}
$\mathtt{Signature}(w,d)$ outputs the degree-d signature of $w$, and has complexity $O(m^d{\ell \choose d}^2)$, where $\ell = |w|$. When $m = 2$ and $d = 1$, it has complexity $O(\ell)$. 
\end{lemma}
\begin{proof}
 From (\ref{eqn:gu.w}), it is straightforward to verify that $g_u^w = L[u]$ if and only if $u$ appears as a scattered subword of $w$. This proves the algorithm's correctness. Now we analyse its complexity. The calculation in Lines 2 and 3 requires $m\ell$ counts, and thus has complexity $O(m\ell)$. Lines 10-13 are executed within three loops, and thus has complexity $O(md{\ell \choose d})$. The last loop computes the convex hull of the support of $g_u^w$ for each $u \in \Sigma^d$. There are $m^d$ such polytopes to compute in the signature (some of which can be empty). 
Each polytope is the convex hull of at most $O({\ell \choose d})$ points, so in the worst case the complexity of computing their convex hull is $O({\ell \choose d}^2)$. Thus, overall, the complexity of the algorithm is 
$O(md{\ell \choose d}) + O(m^d{\ell \choose d}^2) = O(m^d{\ell \choose d}^2)$. For $m = 2$ and $d = 1$, the complexity of computing the support of each polynomial is $O(\ell)$. As these points are naturally sorted, the two polygons in the signature can be computed in linear time using Andrew's monotone chain algorithm \cite{andrew1979efficient}. Thus, the case $m=2,d=1$ has complexity $O(\ell)$.
\end{proof}

\begin{proof}[\textbf{Proof of Theorem \ref{thm:algorithm} for $\mathbf{m > 2}$}]
Fix $n$. Problem \texttt{CheckPair} is solved by Algorithm \ref{alg:identity.general}, which calls $\mathtt{Signature}(w,d)$ for $d = 1, \dots, n-1$. Thus it has complexity of the worst-case, which is that of $\mathtt{Signature}(w,n-1)$. 
Problem \texttt{ListAll} is solved by Algorithm \ref{alg:list.identities.general}, which first computes the $\ut{n}$ signatures for all words in $W(\ell_{a_i}, a_i \in \Sigma)$, and stores these as keys in a hash table. There are ${\ell \choose \ell_{1}, \dots, \ell_{m}}$ such  words, and by Lemma \ref{lem:signature.complexity}, computing the $\ut{n}$ signature in each case has complexity $O(m^{n-1}{\ell \choose n-1}^2)$. Therefore, the complexity is $O(m^{n-1}{\ell \choose \ell_{1}, \dots, \ell_{m}}{\ell \choose n-1}^2)$. Finally, \texttt{ListWord} is solved by Algorithm \ref{alg:equivalence.class.general}, which finds a specific entry from the lookup table created from Algorithm \ref{alg:list.identities.general}, and thus has the same complexity. 
\end{proof}

\begin{algorithm}\caption{\texttt{Signature}} \label{alg:signature}
\flushleft 
\textbf{Input}: Word $w$ on alphabet $\Sigma=\{a_1, \ldots, a_m\}$, positive integer $d$ \\
\textbf{Output}: Degree-$d$ signature of $w$
\begin{algorithmic}[1]
  \State $L \gets $ empty dictionary
  \State $c(0,s) \gets 0$ for all $1 \leq s \leq m$
  \For{$1 \leq i \leq |w|$}
  \State $c(i,s) \gets$number of occurrences of letter $a_s$ before $i$th letter of $w$
  \EndFor
  \For{$\pi \in [{\ell \choose d}]$}
    \State $u \gets (w_{\pi_1}, \dots, w_{\pi_{d}})$ , $u_0\gets 0$, $h \gets \underline{0} \in \mathbb{R}^{md}$
    \For{$1 \leq s \leq m$}
      \For{$k \in \{1, \dots, d\}$}
      \If{$u_{k-1}=a_s$}  \State $N_{s}^w(\pi_{k-1},\pi_k) \gets c(\pi_k,s) - c(\pi_{k-1},s) -1$ 
      \Else \State $N_{s}^w(\pi_{k-1},\pi_k) \gets c(\pi_k,s) - c(\pi_{k-1},s)$
      \EndIf
      \EndFor
  \State     $h \gets h+ N_s^w(\pi_{k-1}, \pi_k)\mathbf{e}_{(k-1)m+s}$
    \EndFor
         \If{$u \notin keys(L)$} $L[u] \gets \{h\}$ \textbf{else }{$L[u] \gets L[u] \cup \{h\}$}
    \EndIf
  \EndFor
\For{$u \in \Sigma^d$ in lexicographic order}
\State $P_u^w \gets$ convex hull of $L[u]$
\EndFor  
  \State \textbf{return} $(P_u^w, u \in \Sigma^d)$ 
\end{algorithmic}
\end{algorithm}
%\vskip-12pt
\begin{algorithm}
\caption{\texttt{Identity}} \label{alg:identity.general}
\flushleft 
\textbf{Input}: words $w,w'$, integer $n \geq 2$ \\
\textbf{Output}: True iff $w \sim_n w'$
\begin{algorithmic}[1]
\For{$d = 1, \dots, n-1$}
  \If{$\mathtt{Signature}(w,d) \neq \mathtt{Signature}(w',d)$} \textbf{return } False
  \EndIf
\EndFor
\State \textbf{return } True
\end{algorithmic}
\end{algorithm}
%\vskip-12pt
\begin{algorithm}
\caption{\texttt{ListIdentitiesGeneral}} \label{alg:list.identities.general}
\flushleft 
\textbf{Input}: $n \geq 2, c \in \mathbb{N}_{\geq 0}^{|\Sigma|} $ \\
\textbf{Output}: All $\ut{n}$ identities for words $w \in W(c)$.
\begin{algorithmic}[1]
\State $L \gets$ empty dictionary
\For{$w \in W(c)$}
  \State $V^w \gets (\mathtt{Signature}(w,1), \dots, \mathtt{Signature}(w,n-1))$
  \State $L[V^w] \pluseq w$
\EndFor
\State \textbf{return} $L$
\end{algorithmic}
\end{algorithm}

\begin{algorithm}
\caption{\texttt{EquivalenceClassGeneral}} \label{alg:equivalence.class.general}
\flushleft 
\textbf{Input}: $n \geq 2$, $\Sigma$ an alphabet of size $m$, $w \in \Sigma^+$ \\
\textbf{Output}: Equivalence class of $w$ in $\ut{n}$
\begin{algorithmic}[1]
\State $c = (|w|_{a_i}, a_i \in \Sigma)$
\State $V^w \gets (\mathtt{Signature}(w,1), \dots, \mathtt{Signature}(w,n-1))$
\State $L = \mathtt{ListIdentitiesGeneral}(n,c)$
\State \textbf{return} $L[V^w]$
\end{algorithmic}
\end{algorithm}

%\vspace{0.2in}
\subsection{Two-letter alphabets}
For a two-letter alphabet and $n = 2$, there is a natural speedup of the equivalence class enumeration algorithm. In particular, thanks to the Structural Theorem, to represent an equivalence class it is sufficient to compute the minimal and maximal element, which can be done very efficiently. This leads to an optimal algorithm for equivalence class computation for $\ut{2}$, and a very efficient algorithm for identity enumeration. As $\ut{n}$ identities must also be $\ut{n-1}$ identities, this also gives a shortcut for these problems for $\ut{n}$ over two-letter alphabets.  

\begin{algorithm}
\caption{\texttt{Minmax}} \label{alg:min.max}
\flushleft 
\textbf{Input}: word $w \in W(\ell_a,\ell_b)$ \\
\textbf{Output}: minimal and maximal words $(\underline{w},\overline{w})$ of the equivalence class of $w$ in $\ut{2}$
\begin{algorithmic}[1]
\State $\overline{w} \gets \mathtt{Maxword}(w)$
\State $\underline{w} \gets \mathtt{Maxword}(\tilde{w})$
\State \textbf{return} $(\underline{w},\overline{w})$. 
\Procedure{Maxword}{$w$}
\State $(A,B) \gets \mathtt{Signature}(w,1)$
\State $\mathcal{V} \gets \mathsf{vertex}(A) \cup \mathsf{vertex}(B) \cup (\ell_a,\ell_b)$, sorted in increasing $\trianglelefteq$ and labelled.
\For{$p_i,p_{i+1} \in \mathcal{V}$}
  \State $\overline{\gamma}(i) \gets \mathtt{MaxPath}(p_i,p_{i+1},\lab(p_i),\lab(p_{i+1}))$
\EndFor
\State Compute $\overline{w}$ from $\overline{\gamma}$ via (\ref{eqn:w.from.a})
\State \textbf{return} $\overline{w}$. 
\EndProcedure
\end{algorithmic}
\end{algorithm}

\begin{algorithm}
\caption{\texttt{EquivalenceClass}} \label{alg:equivalence.class}
\flushleft 
\textbf{Input}: word $w \in W(\ell_a,\ell_b)$, $n \geq 3$ \\
\textbf{Output}: equivalence class of $w$ in $\ut{n'}$ for all $2 \leq n' \leq n$
\begin{algorithmic}[1]
\State $(\underline{w},\overline{w}) \gets \mathtt{Minmax}(w)$
\State $E_1 \gets \{v \in W(\ell_a,\ell_b): \underline{w} \preceq v \preceq \overline{w}\}$. 
\For{$n'$ \textbf{from} $3$ \textbf{to} $n$}
  \State $T \gets \mathtt{Signature}(w,n'-1)$.
  \State $E_{n'-1} \gets \emptyset$
  \For{$v \in E_{n'-2}$}
    \State $S(v) \gets \mathtt{Signature}(v,n' - 1)$
    \If{$S(v) = T$} $E_{n'-1} \pluseq v$
    \EndIf
  \EndFor
\EndFor
\State \textbf{return} $E_n$. 
\end{algorithmic}
\end{algorithm}

\begin{algorithm}
\caption{\texttt{ListIdentities2}} \label{alg:list.identities2}
\flushleft 
\textbf{Input}: $\ell_a,\ell_b \in \mathbb{N}$ \\
\textbf{Output}: All $\ut{2}$ identities for
all words in $W(\ell_a,\ell_b)$
\begin{algorithmic}[1]
\State $L \gets$ empty dictionary
\State $wlist = [a^{\ell_a}b^{\ell_b}]$
\While{$wlist \neq \emptyset$}
  \State $wlist' = \emptyset$
  \For{$w \in wlist$}
    \State $(\underline{w},\overline{w}) \gets \mathtt{Minmax}(w)$
    \If{$\underline{w} \notin keys(L)$}
      \State $L[\underline{w}] \gets (\underline{w},\overline{w})$
      \State $wlist' \gets wlist' \cup \{w': w' \leftrightarrow \underline{w}, w' \succ \underline{w}\}$.
    \EndIf
  \EndFor  
  \State $wlist \gets wlist'$
\EndWhile
\State \textbf{return } $L$
\end{algorithmic}
\end{algorithm}

\begin{lemma}\label{lem:min.max.complexity}
Algorithm \ref{alg:min.max} has complexity $O(\ell)$.
\end{lemma}
\begin{proof}
Step 1 of Procedure \texttt{Maxword} has complexity $O(\ell)$ as noted in Lemma \ref{lem:signature.complexity}. Since the vertices of $A$ and $B$ are already sorted, merging them to a labelled and sorted list $\mathcal{V}$ requires $O(\ell)$ operations. 
Computation of $\overline{\gamma}$ requires one to compute $O(\ell)$ coordinate slices of the lattice polytopes $A(w)$ and $B(w)$, and thus has complexity $O(\ell)$. Finally, reading $\overline{w}$ from $\overline{\gamma}$ requires $O(\ell)$ operations. 
\end{proof}

\begin{proof}[\textbf{Proof of Theorem \ref{thm:algorithm} for $\mathbf{m = 2}$}]
Lemma \ref{lem:signature.complexity} implies that for $n = 2$, Algorithm \ref{alg:signature} has complexity $O(\ell)$. Now consider the problem \texttt{ListWord}. Fix $n \geq 3$. For each $d=2, \ldots, n-1$  Algorithm \ref{alg:equivalence.class} computes the $d$-signature of each word in the $\sim_d$ class of $w$. Thus its complexity is $\sum_{d=2}^{n-1}O(2^d\binom{\ell}{d}^2)C_d(w)$. For $n = 2$, the equivalence class is efficiently represented by the minimal and maximal elements, which can be computed in $O(\ell)$ operations by Lemma \ref{lem:min.max.complexity}. Finally, consider problem \texttt{ListAll}. For $n = 2$, each equivalence class is represented by its min-max pair $(\underline{w},\overline{w})$ in Algorithm \ref{alg:list.identities2}. For each equivalence class, we make at most $\ell$ calls to $\texttt{MinMax}$, which by Lemma \ref{lem:min.max.complexity} has complexity $O(\ell)$. Thus each equivalence class requires at most $O(\ell^2)$ computations. So the overall complexity is $O(C_2\ell^2)$. For $n \geq 3$, the worst-case complexity we give is that of the general-alphabet algorithm. Finally, to see that the algorithms for $\ut{2}$ have optimal order, note that any algorithm would need to read in the word. Reading a word of length $\ell$ already takes $O(\ell)$ operations, and thus the first two algorithms for $\ut{2}$ are optimal. 
\end{proof}

\begin{proof}[\textbf{Proof of Theorem \ref{thm:shortest.ut3}}]
We use Algorithm \ref{alg:list.identities.general} to list all $\ut{3}$ identities of lengths 21 and 22. The algorithm verifies that there are no $\ut{3}$ identities of length 21, and returns the identities in Theorem~\ref{thm:shortest.ut3} as the only identities of length 22. 
\end{proof}

\begin{remark}
 Pastjin \cite{P06} gave an algorithm for identity verification in the bicyclic monoid (and hence in $\ut{2}$ by \cite{DJK17}) via linear programming. Given two words $w,v$, Pastjin's algorithm essentially checks for equality of the polygons $A(w)$ and $A(v)$ by recursively checking if each point $p \in A(w)$ can be separated from $A(v)$ by a hyperplane. In the case of a two letter alphabet, this na\"{i}ve convex hull computation has complexity $O(\ell^2)$, where $\ell$ is the length of the word. In comparison, our algorithm attains the optimal complexity of $O(\ell)$ as shown in Theorem \ref{thm:algorithm}.
\end{remark}

\section{Statistics on $\ut{2}$ and $\ut{3}$: data and conjectures}
\label{sec:alg-results}

In this section we discuss six conjectures obtained by analyzing the data obtained from our algorithms. We hope that these conjectures will fuel new developments at the intersection of semigroup theory, combinatorics and probability.

\subsection{Structural theorem for $\ut{n}$}
On a two-letter alphabet, the Structural Theorem leads to an efficient algorithm to compute the minimal and maximal elements of an equivalence class in $\ut{2}$, and thus resulted in the dramatic speedup for various algorithms for $\ut{2}$. Unfortunately, neither parts of the Structural Theorem does not hold for $\ut{3}$ .
\begin{lemma}\label{lem:not.structural}
There exists words $w,w'$ over a two letter alphabet such that $w \vee w' \sim_3 w \wedge w'$ but $w \not\sim_3 w'$. 
There also exists another pair of words $w,w'$ over a two letter alphabet such that $w \wedge w' \sim_3 w \sim_3 w'$ but $w \not\sim_3 w \vee w'$.
\end{lemma}
\begin{proof}
For the first pair, let $u = baaaabaaaaaababbbbbabbba$, 
$v = babbabbaabbabaa$, and define
$$w = u\, ba\, b \, ab \, v, \quad w' = u\, ab\, b \, ba \, v.  $$
Then $w \vee w' = u\, ba\, b \, ba \, v$ and $w \wedge w' = u\, ab\, b \, ab \, v$. 
For the second pair, let $u = abaaaabbbbaaaabbaabba$, $v = abbaaababbabababbbb$, and define
$$ w = u \, ab \, ba \, v, \quad w' = u \, ba \, ab \, v. $$
Then $w \wedge w' = u \, ba \, ba \, v$, 
$w \vee w' = u \, ab \, ab \, v$.
Figure \ref{fig:not.adjacent}(a) and (b) show the paths and $\ut{2}$ signatures of $w \vee w', w \wedge w'$ for the first and second pair, respectively. It is immediate from this figure that in both examples, all four words $w,w',w \vee w'$ and $w \wedge w'$ are equivalent in $\ut{2}$. However, in the first pair, $w'$ is an isoterm in $\ut{3}$, while the other three are equivalent. In the second example, all but $w \vee w'$ are equivalent in $\ut{3}$.
\end{proof}

In the first example considered in the proof of the previous lemma, the three words lying in the same $\sim_3$ class are connected by adjacent swaps, however, unlike the situation for $\ut{2}$, the order of these swaps is important.  Based on further experiments, we conjecture that the property of being connected via adjacent swaps holds for $\sim_n$ classes more generally. That is, while the Structural Theorem does not hold for $n \geq 3$, the following conjecture states that Corollary \ref{cor:chain} generalises. 
\begin{conjecture}\label{conj:adjacent}
Let $w, v \in \Sigma^+$ with $w \neq v$ and $|\Sigma|=2$, and let $n \geq 2$. Then $w \sim_n v$, if and only if there exists a positive integer $k$ and a sequence of words $u(i)$, for $i=0, \ldots, k$ such that 
$$w=u(0) \leftrightarrow u(1) \leftrightarrow \cdots \leftrightarrow u(k)=v \mbox{ and }
u(0) \sim_n u(1) \sim_n \cdots \sim_n u(k).$$
\end{conjecture}

Conjecture \ref{conj:adjacent} is not true for alphabets of size greater than two, see Example \ref{ex:threeswap}. On the other hand, it holds trivially for the existing families of $\ut{n}$ identities constructed in \cite{I14, TaylorThesis}, as these are of the form $w \sim_n v$ for some $w \leftrightarrow v$.

\begin{figure}[h!]
\begin{subfigure}[t]{0.5\linewidth} 
\centering 
\resizebox{\linewidth}{!}{\includegraphics{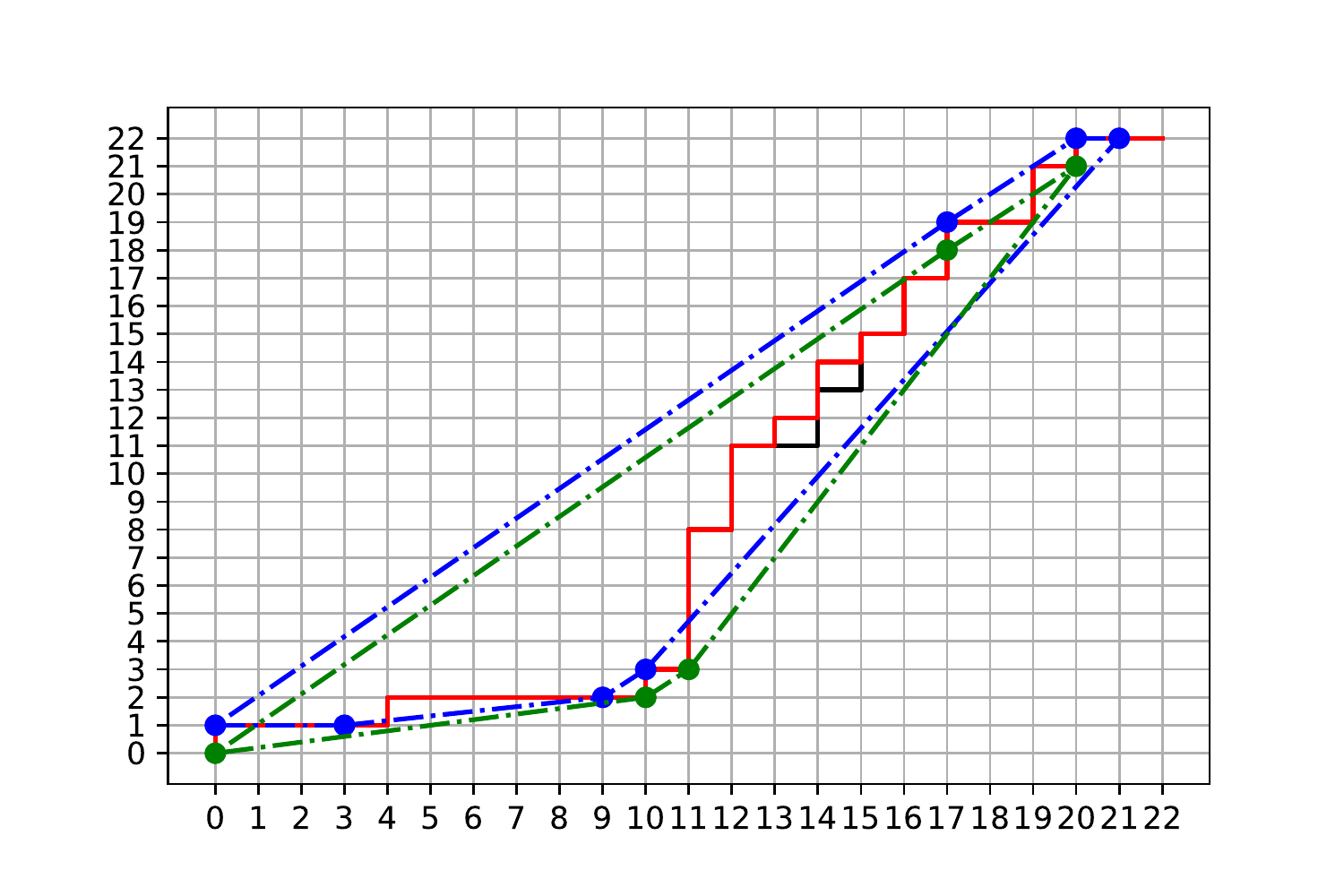}}
\caption{}
\end{subfigure}
\begin{subfigure}[t]{0.45\linewidth} 
\centering 
\resizebox{\linewidth}{!}{\includegraphics[width=0.49\textwidth]{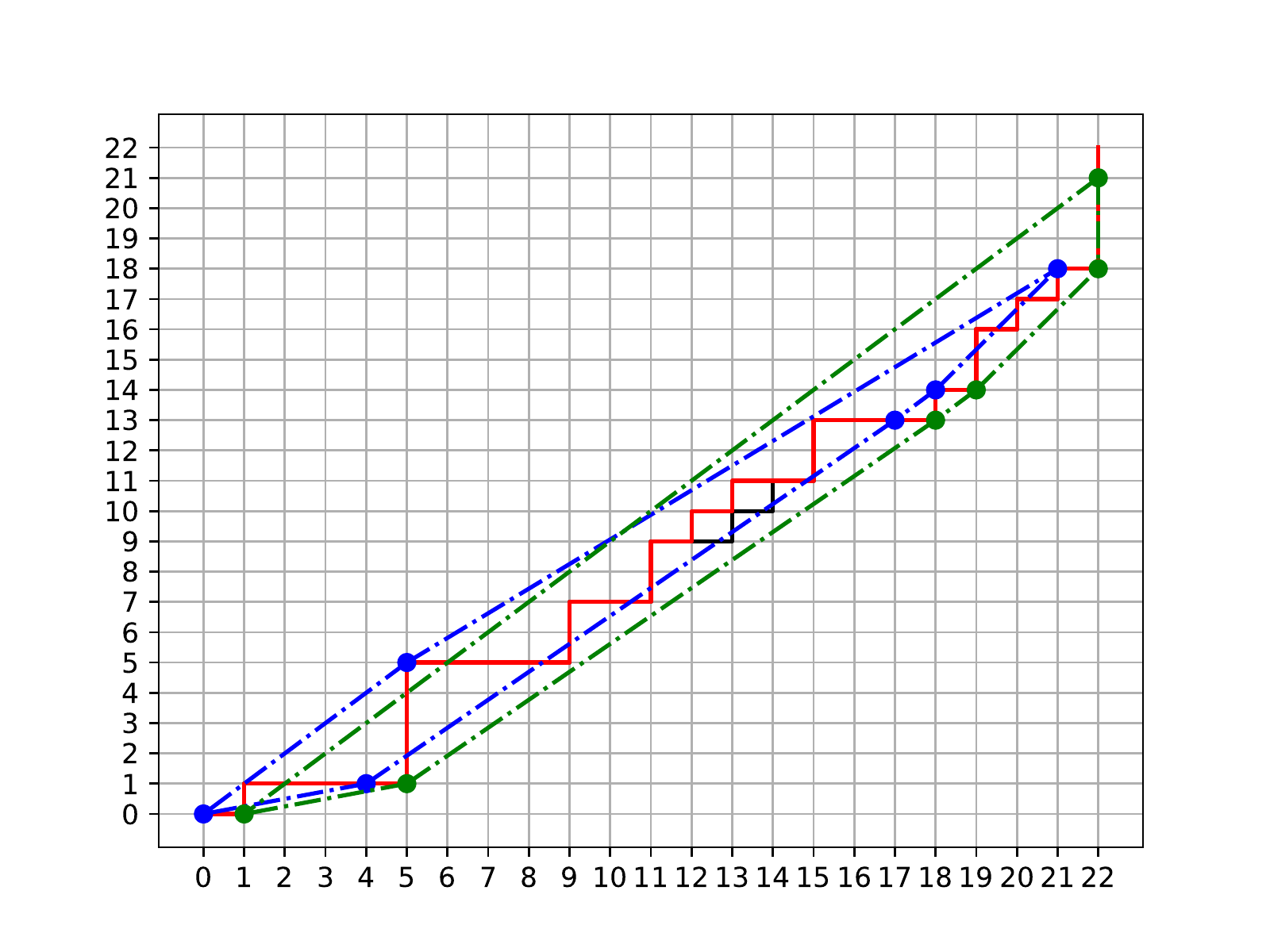}}
\caption{}
\end{subfigure}
\caption{Paths of the words $w \wedge w'$ in black and $w \vee w'$ in red, and their $\ut{2}$ polygons for the examples of Lemma~\ref{lem:not.structural}. These two examples show that neither parts of the Structural Theorem does not hold for $n = 3$.}
\label{fig:not.adjacent}
\end{figure}

One can formulate a weaker version of Conjecture \ref{conj:adjacent} as follows. Say that a word is locally isolated in $\ut{n}$ if $w' \leftrightarrow w$ implies $w' \not\sim_n w$, that is, $w$ does not form $\ut{n}$ identities with any of its neighbours. By Corollary~\ref{cor:chain}, a word over a two letter alphabet is an isoterm for $\ut{2}$ if and only if it is locally isolated. We conjecture that this holds for all $\ut{n}$. 

\begin{conjecture}\label{conj:isoterm.local}
A word over a two letter alphabet is an isoterm for $\ut{n}$ if and only if it is locally isolated in $\ut{n}$.
\end{conjecture}

Verifying whether a word is locally isolated in $\ut{n}$ can be done efficiently, since this amounts to comparing with each of its neighbours, whose number is bounded above by the length of the word. This allows us to estimate the fraction of words in $W(\ell_a,\ell_b)$ that are locally isolated in $\ut{n}$ through random sampling. Since an isoterm for  $\ut{n}$ must be locally isolated, this gives numerical upper bounds for the fraction of $\ut{n}$ isoterms in $W(\ell_a,\ell_b)$. If Conjecture \ref{conj:isoterm.local} holds, then this upper bound is tight.

\begin{figure}[h!]
\includegraphics[height=0.75\textheight]{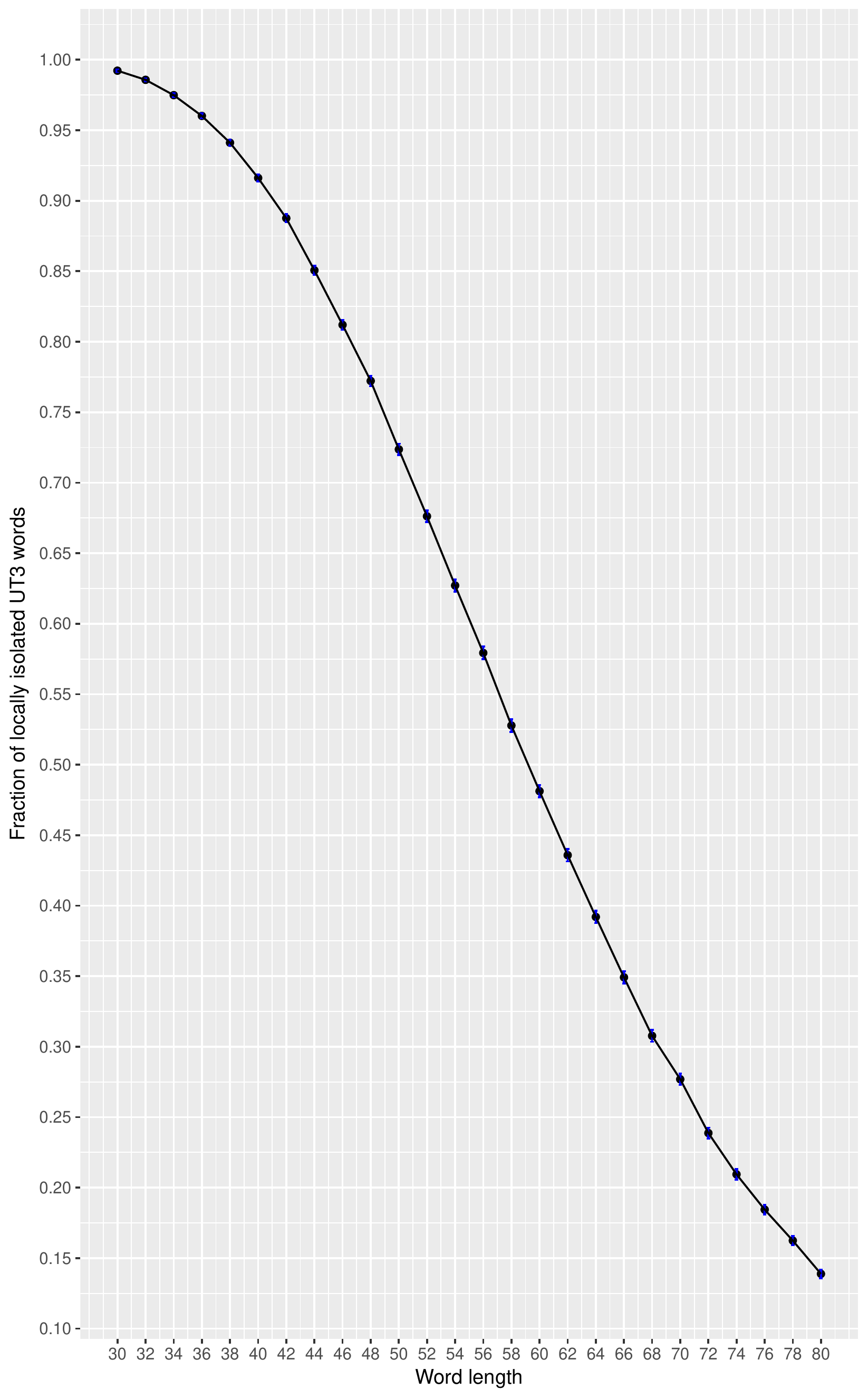}
\caption{Estimated fraction of locally isolated $\ut{3}$ words in $W(\ell/2,\ell/2)$ for $30 \leq \ell \leq 80$ and their $95\%$ confidence intervals shown in blue.}
\label{fig:isoterm-fraction-ut3}
\end{figure}
\begin{example}\label{ex:isoterm.ut3}
We estimate the fraction of locally isolated $\ut{3}$ words in $W(\ell_a,\ell_a)$  for $15 \leq \ell_a \leq 40$. For each value of $\ell_a$, we draw $50\, 000$ independent samples from $W(\ell_a,\ell_a)$, and count the number of locally isolated $\ut{3}$ words. We plot the estimators together with the $95\%$ confidence intervals in Figure \ref{fig:isoterm-fraction-ut3}. Note that the fraction of locally isolated words sharply drops with the word length. Conjecture \ref{conj:isoterm.local} states that this fraction is in fact an unbiased estimate of the true fraction of $\ut{3}$ isoterms. 
\end{example}

Existing $\ut{n}$ identities in the literature have rather intricate and symmetrical constructions \cite{Cain17,I17,TaylorThesis}. In contrast, our experimental findings suggest that for long words, it is extremely easy to randomly generate identities. For instance, if one picks a random word from $W(40,40)$, and just checks amongst its neighbours, Figure \ref{fig:isoterm-fraction-ut3} suggests that the probability that one finds a $\ut{3}$ identity in this way is more than $85\%$.

\begin{conjecture}\label{conj:isoterm.fraction.1}
Fix $r \in (0,1)$ and $n \geq 2$. Let $i(n,\ell,r)$ be the fraction of $\ut{n}$ isoterms in $W(\ell r,\ell(1-r))$. Then $i(n,\ell,r)$ is a monotone decreasing function in $\ell$. Furthermore, $i(n,\ell,r) \to 0$ as $\ell \to \infty$.
\end{conjecture}

\subsection{Fraction of isoterms}

What fraction of the identities which hold in $\ut{2}$ also hold in $\ut{3}$? For short words this fraction is small: for words with length less than 21, for example, it is zero. However, we claim that for long words, a randomly chosen $\ut{2}$ identity has a high chance of being a $\ut{3}$ identity. Specifically, consider the setup of Conjecture~\ref{conj:isoterm.fraction.1}. Fix $r \in (0,1)$, Consider words in $W(\ell r,\ell(1-r))$. Rescale the polygon $\conv(\alpha)$ so that it fits inside a rectangle with side lengths $r \times (1-r)$. For most words, this implies the distance between adjacent points in the rescaled polygon are of order $O(\ell^{-1})$. As $\ell \to \infty$, this has the same effect as finely discretising the polytope $\conv(\alpha)$. For large $\ell$, one would thus expect
$$ \conv(\alpha^w \times \alpha^w \cap C_a) = \conv(\alpha^w \times \alpha^w) \cap C_a $$
for most words $w \in W(\ell r,\ell(1-r))$, 
where $C_a$ is the polyhedron defined in Proposition \ref{prop:transformation}. A similar argument applies to the other polytopes in the degree-two signature. In view of the discussion following Lemma \ref{prop:transformation}, this means for large $\ell$, most $\ut{2}$ identities are also $\ut{3}$ identities. This should especially be true for identities coming from a single adjacent swap.

\begin{conjecture}\label{conj:fraction.ut3}
Fix $r \in (0,1)$ and consider words in $W(\ell r,\ell(1-r))$. Let $w$ be a word chosen uniformly at random in $W(\ell r,\ell(1-r))$, let $u$ be a word chosen uniformly at random from the $\sim_2$-class of $w$, and let $v$ be a word chosen uniformly at random from  amongst the immediate neighbours of $w$ in this class. Then
$\mathbb{P}(w \sim_3 u)$ and $\mathbb{P}(w \sim_3 v)$ are monotone increasing in $\ell$, and they tend to $1$ as $\ell \to \infty.$
\end{conjecture}

\begin{example}
Fix $r = 0.5$. For each $\ell = 200, 400, \dots, 1200$, we draw 50 random words from $W(\ell/2,\ell/2)$. For each word $w$, we compute the ratio of $\ut{3}$-class neighbours to $\ut{2}$-class neighbours, that is, 
$$\frac{|\{w': w' \leftrightarrow w, w' \sim_3 w\}|}{|\{w': w' \leftrightarrow w, w' \sim_2 w\}|}.$$
Figure \ref{fig:longut3-fraction} compares the distribution of these fractions for different $\ell$. In support of Conjecture \ref{conj:fraction.ut3}, this fraction increases as $\ell$ increases.
\begin{figure}[h!]
\centering
\includegraphics[width=\textwidth]{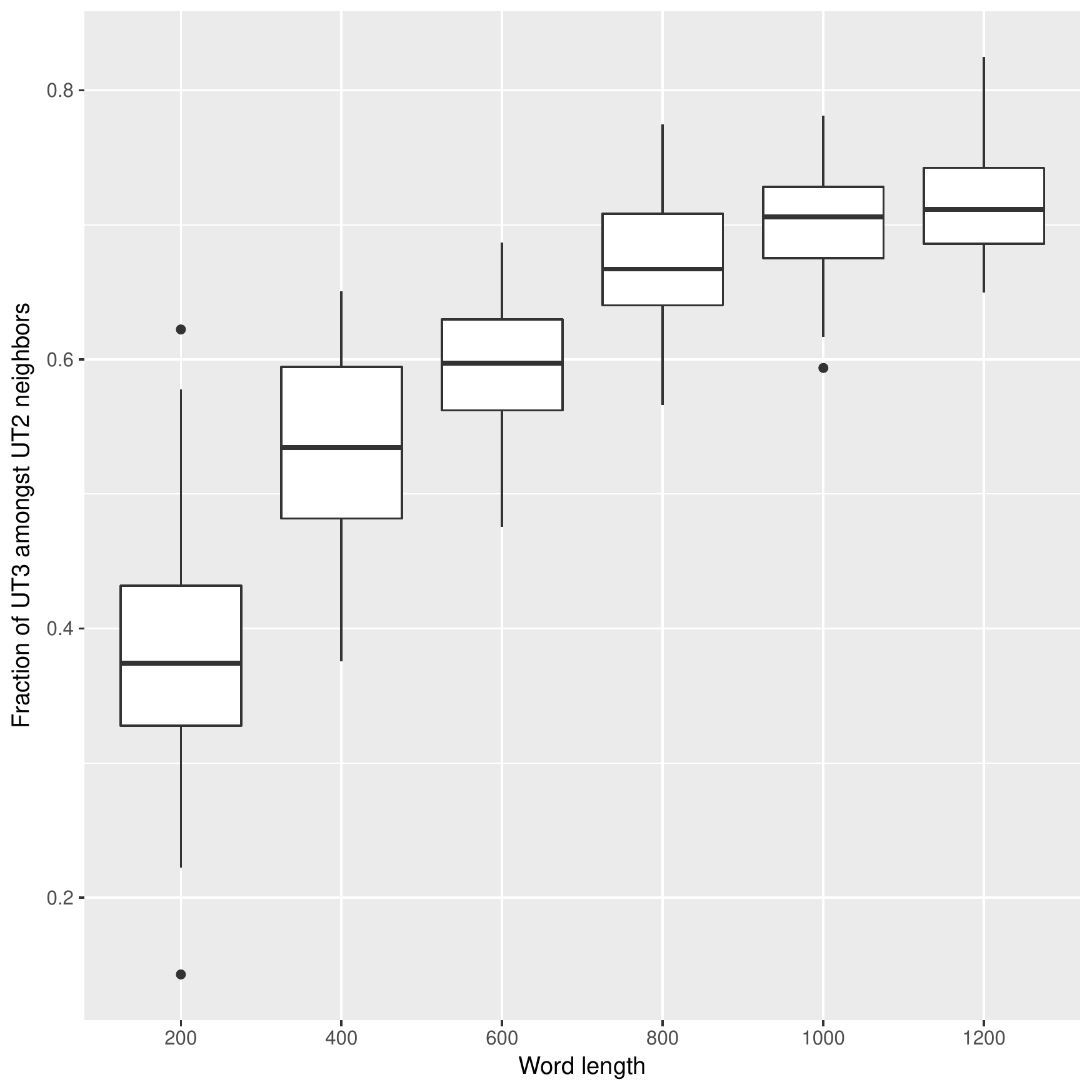}
\caption{Fraction of $\ut{3}$-class neighbours to $\ut{2}$-class neighbours of a randomly chosen word with equal number of $a$'s and $b$'s, for different word lengths.}
\label{fig:longut3-fraction}
\end{figure}
\end{example}

Say that an equivalence class is a twin if it has size 2. Figure \ref{fig:niso-22} below shows the distribution of isoterm, twin, and other classes amongst the $\sim_2$ classes of words of length $22$, as the content of the word varies. The distribution is symmetric as one can swap $a$'s and $b$'s. While the distribution looks approximately binomial, a quick analysis reveals that it is much flatter: it has heavier tail, and the mode is not as high. Note that the distribution of number of words of length $\ell$ as a function of $\ell_a$ is indeed $Binomial(\ell,1/2)$. Thus, should the distribution in Figure \ref{fig:niso-22} be approximately binomial, one should find that the ratio of equivalence classes to total number of words be close to constant. We plot this ratio in Figure \ref{fig:fraction-22} (black line). One finds that this fraction drops sharply as $\ell_a$ approaches $n/2$, see Figure \ref{fig:fraction-22}. A similar phenomenon is observed for the ratio of isoterms relative to the number of equivalence classes (cf. Figure \ref{fig:fraction-22}, blue line). Based on this figure, we conjecture that the fraction of isoterm classes amongst equivalence classes is still asymptotically zero for a fixed letter ratio, as the word lengths tend to infinity. In other words, for long words, it is extremely difficult to construct isoterms through random sampling.

\begin{conjecture}\label{conj:isoterm.fraction.2}
Fix $r \in (0,1)$. Let $e(m,r)$ be the ratio of number of $\ut{2}$ isoterms to the number of $\ut{2}$ equivalence classes in $W(\ell r,\ell(1-r))$. Then $e(m,r)$ is a monotone non-increasing function in $\ell$. Furthermore, $e(m,r) \to 0$ as $\ell \to \infty$.
\end{conjecture}

\begin{figure}[h]
\includegraphics[width=\textwidth]{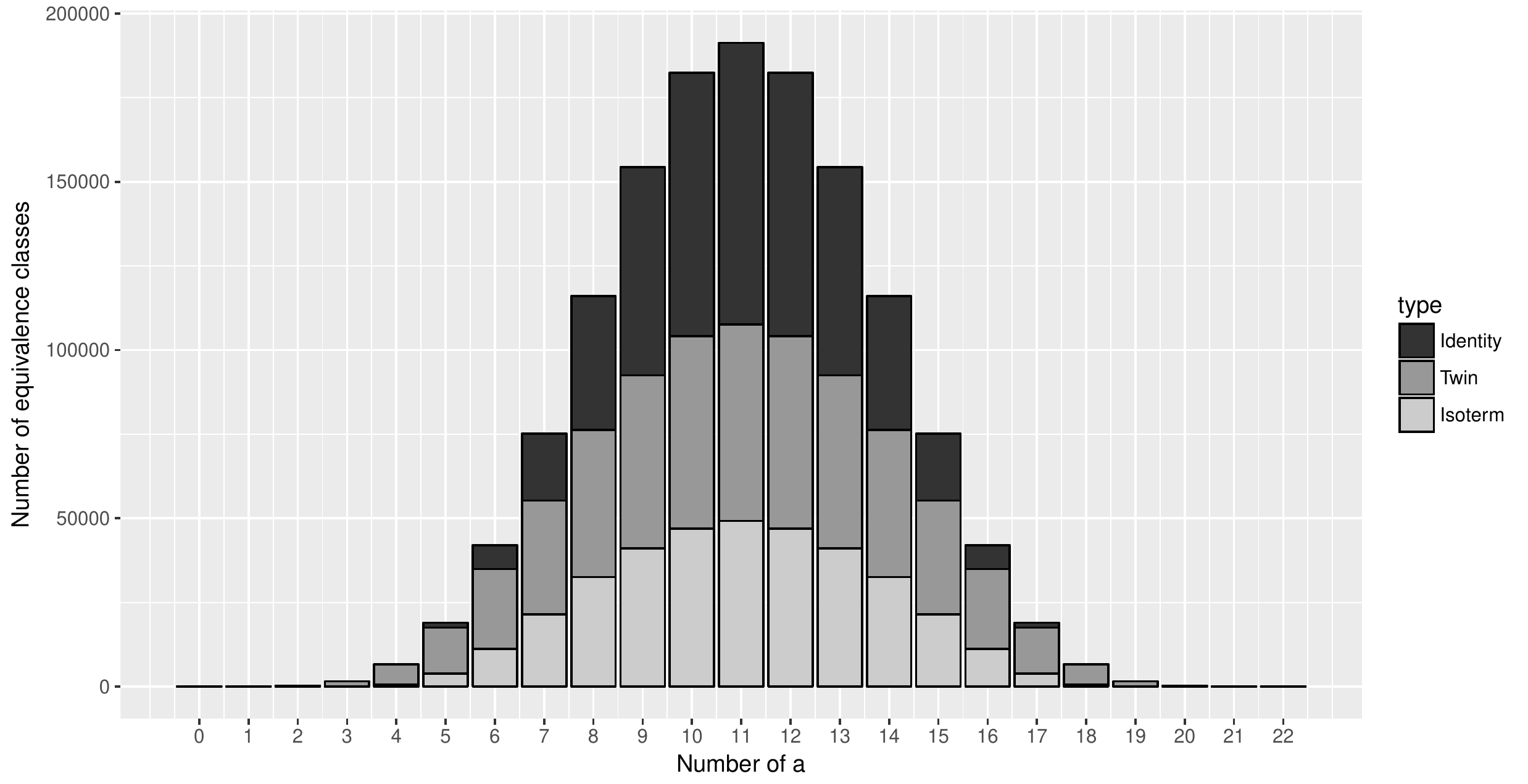}
\caption{Number of isoterms, twins and non-twin classes of identities in $\ut{2}$ for words with prescribed number of $a$'s and of length $22$.}\label{fig:niso-22}
\end{figure}

\begin{figure}[h]
\includegraphics[width=0.75\textwidth]{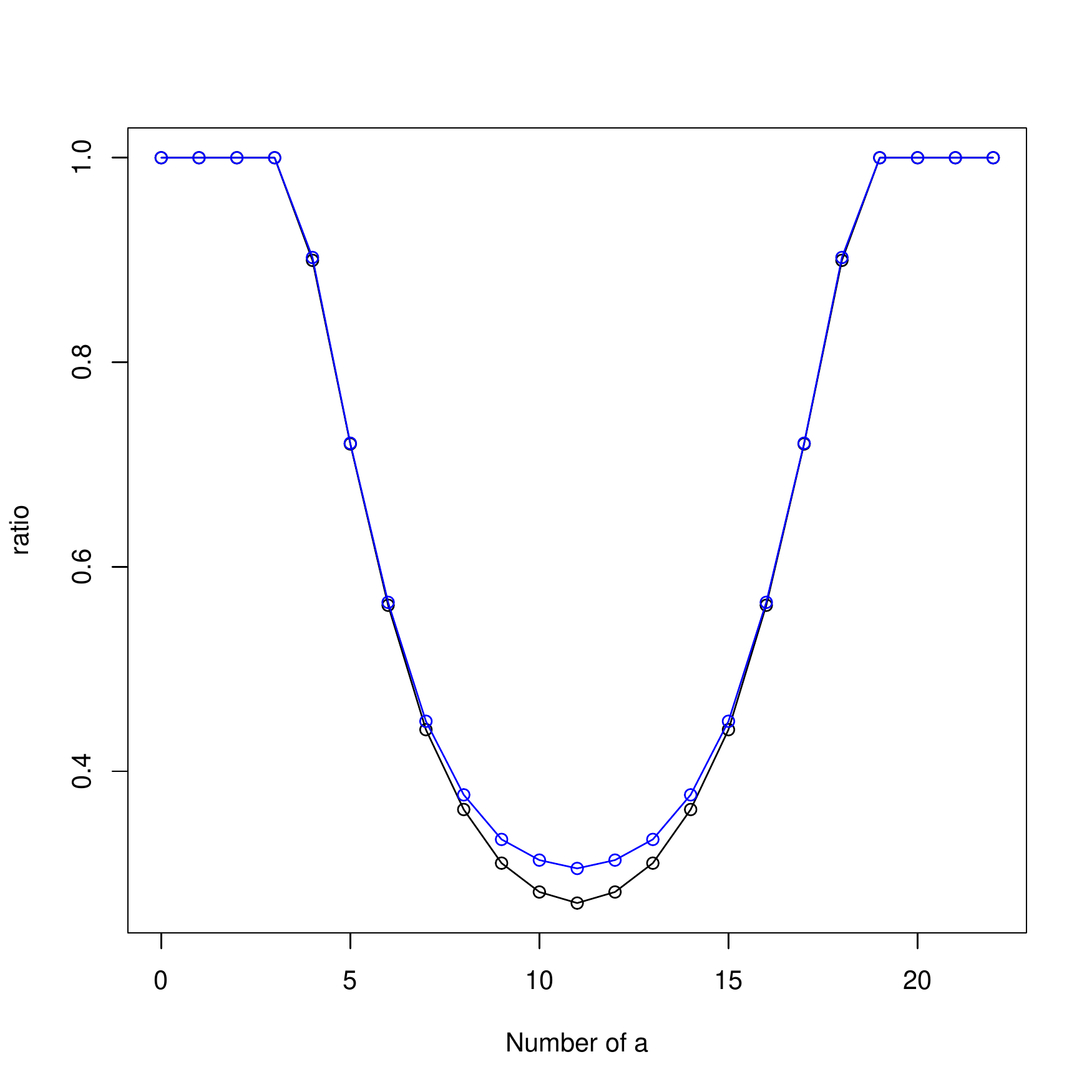}
\caption{The ratio of $\ut{2}$ classes to words  in black, and $\ut{2}$ isoterms to classes in blue, as a function of $\ell_a$, for binary words of length $22$.}\label{fig:fraction-22}
\end{figure}

We conclude with a conjecture of a combinatorial nature, concerning the size of the largest equivalence class amongst words with fixed number of letters. This conjecture has been verified numerically through exhaustive enumeration up to words of length~24.

\begin{conjecture}
\label{conj:large}
For each fixed $\ell_a \geq 5$, the largest $\ut{2}$ equivalence class in $W(\ell_a,\ell_a)$ belongs to the generalised Catalan family defined in Example~\ref{ex:catalan}.
\end{conjecture}

\section{Summary}
\label{conclusion}
This work gives new geometric methods and algorithms to verify, construct and enumerate identities which hold in the semigroup of upper-triangular $n \times n$ tropical matrices $\ut{n}$.  Each word over a fixed alphabet $\Sigma$ has an associated signature, which is a sequence of polytopes, and two words over the same alphabet form an identity in $\ut{n}$ if and only if the first $\sum_{d=1}^{n-1} |\Sigma|^d$ terms of their signature agree. This translates problems concerning semigroup identities in $\ut{n}$ to questions about polytopes. With this insight, we obtained the Structural Theorem for $\ut{2}$ identities, a short proof of Adjan's Theorem for the bicyclic monoid and explicit construction of the shortest $\ut{3}$ identities. Furthermore, we show that the signature of a word can be efficiently computed. This yields algorithms to verify and enumerate $\ut{n}$ identities. We implement these algorithms, and produce six interesting conjectures and a wealth of data that support them. These conjectures are at the intersection of combinatorics, semigroup theory and probability. Our results and conjectures call for new research direction at this interface.

\bibliography{refs}
\bibliographystyle{plain}
\end{document}